\documentclass[11pt]{article}
\pdfoutput=1
\usepackage{amsfonts,amsmath,amssymb,amsthm}
\usepackage[pdftex]{graphicx}
\usepackage[hmargin=1in,vmargin=1in]{geometry}
\usepackage[sort&compress]{natbib}				%for citations in Harvard style
\usepackage{appendix}
\usepackage{footnote}
\usepackage{multirow}
\usepackage[colorlinks,citecolor=blue]{hyperref}
\usepackage{enumerate}
\usepackage{caption}
\usepackage[finnish, british]{babel}
\usepackage{capt-of}
\usepackage{soul}
\usepackage{hyperref}
\usepackage{float}
\usepackage{booktabs, threeparttable}
\usepackage{multibib}
\newcites{supplement}{Supplemental References}

\def\qed{\rule{2mm}{2mm}}

\let\footnote=\endnote

\usepackage{tikz,graphicx,pgfplots,pgfkeys,subcaption}   %The subcaption package is used in place of "subfig"
\def\addlegendimage{\csname pgfplots@addlegendimage\endcsname}

\usepackage[pagewise,mathlines]{lineno}
\mathchardef\dash="2D
\newtheorem{theorem}{Theorem}[section]
\newtheorem{lemma}{Lemma}[section]

\newtheorem{corollary}{Corollary}[section]
\newtheorem{proposition}{Proposition}[section]
\theoremstyle{assumption}
\newtheorem{assumption}{Assumption}[section]
\theoremstyle{definition}

\newtheorem{example}{Example}[section]
\newtheorem{remark}{Remark}[section]
\usepackage{etoolbox} % This package goes here to add \qed to remarks automatically.
\AtEndEnvironment{remark}{~\qed}%
\AtEndEnvironment{example}{~\qed}%

\begin{document}
%\date{}

\author{
Eric Mbakop\\
Department of Economics\\
University of Calgary\\
\url{eric.mbakop@ucalgary.ca}
\and
Max Tabord-Meehan\\
Department of Economics\\
University of Chicago\\
\url{maxtm@uchicago.edu}
\footnote{We are grateful for advice and encouragement from Ivan Canay, Joel Horowitz, Chuck Manski and Alex Torgovitsky. We would also like to thank Toru Kitagawa, Azeem Shaikh, Alex Tetenov, Stefan Wager, the editor, anonymous referees, and seminar participants at Northwestern University, (continued on next page)
}\\
}
\title{Model Selection for Treatment Choice: \\Penalized Welfare Maximization}

\maketitle
\vspace{-8.8mm}
\begin{abstract}
This paper studies a penalized statistical decision rule for the treatment assignment problem. Consider the setting of a utilitarian policy maker who must use sample data to allocate a binary treatment to members of a population, based on their observable characteristics. We model this problem as a statistical decision problem where the policy maker must choose a subset of the covariate space to assign to treatment, out of a class of potential subsets. We focus on settings in which the policy maker may want to select amongst a \emph{collection} of constrained subset classes: examples include choosing the number of covariates over which to perform best-subset selection, and model selection when approximating a complicated class via a sieve. We adapt and extend results from statistical learning to develop the Penalized Welfare Maximization (PWM) rule. We establish an oracle inequality for the regret of the PWM rule which shows that it is able to perform model selection over the collection of available classes. We then use this oracle inequality to derive relevant bounds on maximum regret for PWM. An important consequence of our results is that we are able to formalize model-selection using a ``hold-out" procedure, where the policy maker would first estimate various policies using half of the data, and then select the policy which performs the best when evaluated on the other half of the data.
\end{abstract}

\noindent KEYWORDS: Treatment Choice, Minimax-Regret, Statistical Learning \\
\noindent JEL classification codes: C01, C14, C44, C52

\clearpage
\section{Introduction} \label{sec:intro}
\footnotetext[0]{(continued from previous page), NASMES 2017, and the Bristol Econometrics Study Group for helpful comments, as well as Nitish Keskar for help in implementing EWM. This research was supported in part through the computational resources and staff contributions provided for the Social Sciences Computing Cluster (SSCC), and the Quest high performance computing facility at Northwestern University. All mistakes are our own.}

This paper develops a new statistical decision rule for the treatment assignment problem. A major goal of treatment evaluation is to provide policy makers with guidance on how to assign individuals to treatment, given experimental or quasi-experimental data. Following the literature inspired by \cite{manski2004} \citep[a partial list in econometrics includes][]{dehejia2005, schlag2007, hirano2009, stoye2009, chamberlain2011, bhat2012, tetenov2012, stoye2012, athey2017, kt2015, armstrong2015, kock2017, rai2018, viviano2019}, we treat the treatment assignment problem as a statistical decision problem of maximizing population welfare. Like many of the above papers, we evaluate our decision rule by its maximum regret. 

The rule we develop, the Penalized Welfare Maximization (PWM) rule, is designed to address situations in which the policy maker can choose amongst a \emph{collection} of constrained classes of allocations. To be concrete, suppose we have two treatments, and we represent assignment into these treatments by partitioning the covariate space into two pieces. We can then think of constraints on assignment as constraints on the allowable subsets that we can consider for the partitions. For example, policy makers may face exogenous constraints on how they can use covariates for legal, ethical, or political reasons. Even in cases where policy makers have leeway in how they assign treatment, plausible modeling assumptions may imply certain restrictions on assignment. \cite{kt2015} develop what they call the Empirical Welfare Maximization (or EWM) rule, whose primary feature is its ability to solve the treatment choice problem when certain exogenous constraints are placed on assignment. \cite{kt2015} focus on deriving bounds on maximum regret of the EWM rule for a \emph{fixed} class of subsets of finite VC dimension  \citep[see][for a definition]{devroye1996}. In this paper, however, we consider settings where the class of allowable subsets is ``large". We approach the problem by approximating our class of allowable allocations by a sequence of subclasses of finite VC dimension. We establish an oracle inequality for the regret of the PWM rule which shows that it behaves as if we knew the ``correct" class to use in the sequence. We then use this result to derive bounds on the maximum regret of the PWM rule in two empirically relevant settings.

The main setting that we consider is one where the class of feasible allocations has infinite VC dimension. In particular, we argue that economic modeling assumptions may sometimes put restrictions on the unconstrained optimum that naturally generate classes of infinite VC dimension. For example, plausible assumptions may only impose shape restrictions on the optimal allocation. To solve the optimal welfare assignment problem in this setting, we approximate these large classes of feasible allocations by sequences of classes of finite VC dimension. The strength of the PWM rule in this setting will then be to provide a data-driven method by which to select an ``appropriate" approximating class. In doing so we will derive bounds on the maximum regret of the PWM rule for a large set of classes of infinite VC dimension. 

We also consider the setting where the class of feasible allocations may have large VC dimension relative to the sample size. This could arise, for example, if the planner has many covariates on which to base assignment. As is shown in \cite{kt2015}, when the constraints placed on assignment are too flexible relative to the sample size available, the EWM rule may suffer from overfitting, which can result in inflated values of regret. By the same mechanism that allows PWM to select an appropriate approximating class in our first application, we can use PWM in order to select amongst simpler subclasses in this setting as well, in a way that improves the performance of the allocation rule in finite samples. 

The PWM rule is heavily inspired by the literature on model selection in classification: see for example the seminal work of \cite{vapnik1974}, as well as \cite{devroye1996}, \cite{kolt2001}, \cite{bartlett2002}, \cite{boucheron2005}, \cite{scott2006},  \cite{bartlett2008}, \cite{kolt2008} among many others. The theoretical contribution of our paper is to modify and extend some of these tools to the setting of treatment choice. In deciding which tools to extend, we have attempted to strike a balance between ease of use for practitioners, theoretical appeal, and performance in simulations. An important consequence of our results is that we are able to formalize model-selection using a ``hold-out" procedure, where the policy maker would first estimate various policies using half of the data, and then select the policy which performs the best when evaluated on the other half of the data. The connection between classification and treatment choice has been explored in various fields, including machine learning, under the label of \emph{policy learning} \citep[see][among others]{zadrozny2003, beygelzimer2009, swaminathan2015, kallus2016}, and in epidemiology under the label of \emph{individualized treatment rules}  \citep[examples include][]{qian2011, zhao2012}. \cite{kt2015} and \cite{athey2017} provide a discussion on the link between these various literatures.

The remainder of the paper is organized as follows. In Section \ref{sec:setup}, we set up the notation and formally define the problem that the policy maker (i.e. social planner) is attempting to solve. In Section \ref{sec:results}, we introduce the PWM rule, present general results about its maximum regret, and explain how our results allow us to study the properties of the ``hold-out" model selection procedure. In Section \ref{sec:application} we derive bounds on maximum regret of the PWM rule when the planner is constrained to what we call \emph{monotone} allocations, and then apply PWM in an application using data from the Job Training Partnership Act (JTPA) study. 

\section{Setup and Notation} \label{sec:setup}
Let $Y_i$ denote the observed outcome of a unit $i$, and let $D_i$ be a binary variable which denotes the treatment received by unit $i$. Let $Y_i(1)$ denote the potential outcome of unit $i$ under treatment $1$ (which we will refer to as ``the treatment''), and let $Y_i(0)$ denote the potential outcome of unit $i$ under treatment $0$ (which we will refer to as ``the control"). The observed outcome for each unit is related to their potential outcomes through the expression:
\begin{equation} \label{eq:outcomeid}
Y_i = Y_i(1)D_i + Y_i(0)(1-D_i)~.
\end{equation}
Let $X_i\in \mathcal{X}\subset \mathbb{R}^{d_x}$ denote a vector of observed covariates for unit $i$. Let $Q$ denote the distribution of $(Y_i(0),Y_i(1),D_i,X_i)$, then we assume that the planner observes a size $n$ random sample
$$(Y_i, D_i, X_i)_{i=1}^n\sim P^n~,$$ where $P$ is jointly determined by $Q$, and the expression in (\ref{eq:outcomeid}). Throughout the paper we will assume unconfoundedness, i.e.
\begin{assumption}(Unconfoundedness)\label{ass:so} The distribution $Q$ satisfies:
$$\Big((Y(1),Y(0)) \perp D\Big) \hspace{1mm} \bigg| X~.$$
\end{assumption}

This assumption asserts that, once we condition on the observable covariates, the treatment is exogenous. This assumption will hold in a randomized controlled trial (RCT), which is our primary application of interest, since the treatment is exogenous by construction. 

The planner's goal is to optimally assign the treatment to the population. The objective function we consider is utilitarian welfare, which is defined by the average of the individual outcomes in the population:
$$E_Q[Y(1) {\bf 1}\{X \in G\} + Y(0) {\bf 1}\{X \notin G\}]~,$$
where $G\subset \mathcal{X}$ represents the set of covariate values of the individuals assigned to treatment. The planner is tasked with choosing a \emph{treatment allocation} $G \subset \mathcal{X}$ using the empirical data. Using Assumption \ref{ass:so}, we can rewrite the welfare criterion as:
$$E_Q[Y(0)] + E_P\Big[\Big(\frac{YD}{e(X)} - \frac{Y(1-D)}{1-e(X)}\Big){\bf 1}\{X \in G\}\Big]~,$$
where $e(X)=E_P[D|X]$ is the propensity score. Since the first term of this expression does not depend on $G$, we define the planner's objective function given a choice of treatment allocation $G$ as:
$$W(G) := E_P\Big[\Big(\frac{YD}{e(X)} - \frac{Y(1-D)}{1-e(X)}\Big){\bf 1}\{X \in G\}\Big]~.$$

Let $\mathcal{G}$ be the class of all feasible treatment allocations. Here, we consider the possibility that the planner may be restricted in what type of allocations she can (or wants to) consider. These restrictions may arise from legal, ethical, or political considerations, or could arise as natural constraints from an economic model. Consider the following three examples of $\mathcal{G}$:

\begin{example}\label{ex:treatalloc1}
$\mathcal{G}$ could be the set of all measurable subsets of $\mathcal{X}$. This is the largest possible class of admissible allocations. It is straightforward to show that the optimal allocation in this case is as follows: define 
$$\tau(x) := E_Q[Y(1) - Y(0)|X=x]~,$$
then the optimal allocation is given by
$$G^*_{FB} := \{x \in \mathcal{X} : \tau(x) \ge 0\}~,$$
which assigns an individual with covariate $x$ to treatment or control depending on whether the conditional average treatment effect  at $x$ is non-negative.
\end{example}
\begin{example}\label{ex:treatalloc2}
Suppose $\mathcal{X}  \subset \mathbb{R}$, and consider the class of \emph{threshold allocations}:
$$\mathcal{G}  = \{G : G = (-\infty, x] \cap \mathcal{X} \hspace{2mm} \text{or} \hspace{2mm} G = [x, \infty) \cap \mathcal{X}, \hspace{2mm} \text{for} \hspace{2mm} x \in \mathcal{X} \}~.$$
Such a class $\mathcal{G}$ would be reasonable, for example, when assigning scholarships to students: suppose the only covariate available to the planner is a student's GPA, then it may be school policy that only threshold-type rules are to be considered.
\end{example}
\begin{example}\label{ex:treatalloc3}
Let $\mathcal{X} = \mathcal{X}_1\times\mathcal{X}_2  \subset \mathbb{R}^2$, and consider the class of \emph{monotone allocations}:
$$\mathcal{G} = \big\{G : G = \{(x_1,x_2) \in \mathcal{X} | \hspace{1mm} x_2 \ge f(x_1) \hspace{1mm} \text{for} \hspace{1mm} f:\mathcal{X}_1 \rightarrow \mathcal{X}_2 \hspace{1mm} \text{increasing}\} \big\}~.$$

As an example, consider again the setting of assigning scholarships to students (Example \ref{ex:treatalloc2}), but now suppose that the covariates available to the planner are parental income ($x_1$) and a student's GPA ($x_2$). The allocation rules considered in $\mathcal{G}$ are such that the GPA requirement for scholarship eligibility increases with parental income. Such a restriction could be imposed exogenously or could potentially arise as a shape restriction from an economic model.
\end{example}

Given a feasible class $\mathcal{G}$, we denote the highest attainable welfare by:
$$W^*_\mathcal{G} := \sup_{G \in \mathcal{G}} W(G)~.$$
A \emph{decision rule} is a function $\hat{G}$ from the observed data $\{(Y_i,D_i,X_i)\}_{i=1}^n$ into the set of admissible allocations $\mathcal{G}$.  We call the rule that we develop and study in this paper the \emph{Penalized Welfare Maximization} (or PWM) rule. As in much of the literature that follows the work of \cite{manski2004}, we assume that the planner is interested in rules $\hat{G}$ that, on average, are close to the highest attainable welfare. To that end, the criterion by which we evaluate a decision rule is given by what we call maximum $\mathcal{G}$\emph{-regret}:
$$\sup_P E_{P^n}[W^*_\mathcal{G} - W(\hat{G})]~.$$

\section{Penalized Welfare Maximization} \label{sec:results}
In this section, we present the main results of our paper. In Section \ref{ss:EWM}, we review some properties of the empirical welfare maximization (EWM) rule of \cite{kt2015}, which will motivate the PWM rule and serve as an important building block in its construction. In Section \ref{ss:PWMres}, we define the penalized welfare maximization rule and present bounds on its maximum $\mathcal{G}$-regret for general penalties. In Section \ref{ss:PWMpen} we illustrate these results by applying them to some specific penalties, and in particular we show that a standard ``hold-out" procedure can be formalized as a penalty which satisfies our assumptions. In Section \ref{ss:PWMe} we present results for a modification of the PWM rule for quasi-experimental settings where the propensity score is not known and must be estimated. 

\subsection{Empirical Welfare Maximization: a Review and Some Motivation}\label{ss:EWM}
The EWM rule solves a sample analog of the population welfare maximization problem:
$$\hat{G}_{EWM} \in \arg\max_{G \in \mathcal{G}} W_n(G)~,$$
where
\begin{equation} \label{eq:EWMobjective}
W_n(G) := \frac{1}{n}\sum_{i=1}^n \tau_i {\bf 1}\{X_i \in G\} := \frac{1}{n}\sum_{i=1}^n\Big[\Big(\frac{Y_iD_i}{e(X_i)} - \frac{Y_i(1-D_i)}{1-e(X_i)}\Big){\bf 1}\{X_i \in G\}\Big]~.
\end{equation}
\cite{kt2015} show how to formulate this problem as a Mixed Integer Linear Program (MILP) for many classes $\mathcal{G}$ of practical interest (see Appendix C for examples). Alternatively, direct parameter search has been shown to be very effective at solving the welfare maximization problem in certain applications as well: see for example \cite{zhou2018}. Note that to solve this optimization problem, the planner must know the propensity score $e(\cdot)$. This assumption is reasonable if the data comes from a randomized experiment, but clearly could not be made in a setting where the planner is using observational data. \cite{kt2015} derive results for a modified version of the EWM rule where the propensity score is estimated, which we will review in Section \ref{ss:PWMe}.

To derive their non-asymptotic bounds on the maximum $\mathcal{G}$-regret of the EWM rule, \cite{kt2015} make the following assumptions, which we will also maintain in our results:

\begin{assumption}(Bounded Outcomes and Strict Overlap)\label{ass:BO} The set of distributions $\mathcal{P}(M,\kappa)$ has the following properties:
\begin{itemize}
\item There exists some $M < \infty$ such that the support of the outcome variable $Y$ is contained in $[-\frac{M}{2}, \frac{M}{2}]$.
\item There exists some $\kappa \in (0,0.5)$ such that $e(x) \in [\kappa, 1-\kappa]$ for all $x$.
\end{itemize}
\end{assumption}

In order to derive their results, \cite{kt2015} also make the following assumption, which we will \emph{not} require:

\begin{assumption}(Finite VC Dimension):\label{ass:VCdim}
$\mathcal{G}$ has finite VC dimension $V < \infty$.
\end{assumption}

Such an assumption may or may not be restrictive depending on the application in question. Consider Example \ref{ex:treatalloc2}, the class of threshold allocations on $\mathbb{R}$. This class has VC dimension 2, thus Assumption \ref{ass:VCdim} holds. On the other hand, it can be shown that the class of monotone allocations on $[0,1]^2$ that was introduced in Example \ref{ex:treatalloc3} has infinite VC dimension \citep[see][]{devroye1996}.

Given Assumptions \ref{ass:BO} and \ref{ass:VCdim}, \cite{kt2015} derive the following non-asymptotic upper bound on the maximum $\mathcal{G}$-regret of the EWM rule:
\begin{equation} \label{eq:EWMregretbound}
\sup_{P \in \mathcal{P}(M,\kappa)} E_{P^n} [W^*_\mathcal{G} - W(\hat{G}_{EWM})] \le C\frac{M}{\kappa}\sqrt{\frac{V}{n}}~,
\end{equation}
for some universal constant $C$. Moreover, when $X$ has sufficiently ``large" support, they derive the following \emph{lower} bound: for \emph{any} decision rule $\hat{G}$,
\begin{equation} \label{eq:EWMlowerbound}
\sup_{P \in \mathcal{P}(M,\kappa)} E_{P^n}[W^*_\mathcal{G} - W(\hat{G})] \ge RM\sqrt{\frac{V-1}{n}}~,
\end{equation}
for $R$ a universal constant and for all sufficiently large $n$. This shows that the rate of convergence of maximum $\mathcal{G}$-regret implied by (\ref{eq:EWMregretbound}) is the best possible, i.e. that no other decision rule could achieve a faster rate without imposing additional assumptions.

\begin{remark}\label{rem:infVC}
Theorem 2.2 in \cite{kt2015}, which establishes (\ref{eq:EWMlowerbound}), implies that if $X$ has ``large" support and we do not impose additional restrictions on the set of distributions, then it is \emph{impossible} to derive a uniform rate of convergence of maximum $\mathcal{G}$-regret for \emph{any} rule, for classes $\mathcal{G}$ of infinite VC dimension. A related result is derived in \cite{stoye2009}, who shows that in a setting with a continuous covariate, and for any sample size, flipping a coin to assign individuals to treatment is minimax-regret optimal. Hence we will require additional restrictions on the set of distributions when deriving bounds on maximum regret for classes $\mathcal{G}$ of infinite VC dimension.
\end{remark}

\subsection{Penalized Welfare Maximization: General Results}\label{ss:PWMres}
We now consider a setting where the class $\mathcal{G}$ of admissible rules is ``large", but can be ``approximated" by a sequence of less complex subclasses $\mathcal{G}_k$:\footnotemark[1]
$$\mathcal{G}_1 \subseteq \mathcal{G}_2 \subseteq \mathcal{G}_3\subseteq \cdots \subseteq \mathcal{G}_k\subseteq \cdots \subseteq \mathcal{G}~.$$
\footnotetext[1]{As can be seen from the proofs, the results we present below remain valid even if the sequence $\{\mathcal{G}_k\}_k$ is not nested.}
Let $\hat{G}_{n,k}$ be the EWM rule in the class $\mathcal{G}_k$. Then we can decompose the $\mathcal{G}$-regret of the rule $\hat{G}_{n,k}$ as follows:
$$E_{P^n}[W^*_{\mathcal{G}} - W(\hat{G}_{n,k})] = E_{P^n}[W^*_{\mathcal{G}_{k}} - W(\hat{G}_{n,k})] + W^*_{\mathcal{G}} - W^*_{\mathcal{G}_{k}}~.$$ 
Given this decomposition, we call
$$E_{P^n}[W^*_{\mathcal{G}_{k}} - W(\hat{G}_{n,k})]~,$$
the \emph{estimation} error of the rule $\hat{G}_{n,k}$ in the class $\mathcal{G}_k$, and we call
$$W^*_{\mathcal{G}} - W^*_{\mathcal{G}_{k}}~,$$
the \emph{approximation} error (or bias) of the class $\mathcal{G}_k$. Note that since the classes $\{\mathcal{G}_k\}_k$ are nested, the estimation error (respectively approximation bias) is non-decreasing (resp. non-increasing) with respect to $k$. If one had sharp uniform bounds on these errors, then an appropriate choice of $k$ would be one that minimizes the sum of these bounds. In Theorem \ref{thm:regretbound}, we derive an oracle inequality which shows that PWM selects such a $k$, in a data-driven fashion. We use this feature of PWM to derive bounds on maximum regret in two settings of empirical interest.

The first setting we consider is one where $\mathcal{G}$ has infinite VC dimension (see Examples \ref{ex:treatalloc1} and \ref{ex:treatalloc3}). In this setting, performing EWM on the whole class $\mathcal{G}$ may be undesirable. For example, regret may not converge to zero, or may converge to zero at a suboptimal rate \citep[see][for related results in a regression context]{birge1993}, or it may simply be the case that maximization over $\mathcal{G}$ is computationally difficult. Instead, we apply EWM to an approximating class $\mathcal{G}_k$, and we allow the complexity of the approximating class to grow as the sample size increases. We present examples of relevant approximating classes in Examples \ref{ex:aseq3} and \ref{ex:aseq1} below. In Corollary \ref{cor:regretrate} we establish a bound on maximum regret in this setting.

The second setting that we consider is one where the class $\mathcal{G}$ has finite but large VC dimension relative to the sample size. This situation can arise, for instance, in applications where the planner has a large set of covariates on which to base treatment, and where the feasible allocations are threshold allocations (see Example \ref{ex:aseq2} below). The bound on regret given by (\ref{eq:EWMregretbound}) increases with the VC dimension $V$ of $\mathcal{G}$, so that EWM tends to ``overfit"' the data when $V$ is large relative to the sample size. In this situation, it may be beneficial to perform EWM in a subclass $\mathcal{G}'$ of smaller VC dimension, and hence we face the same tradeoff between estimation and approximation error that was noted above. In Corollary \ref{cor:regretpt} we specialize Theorem \ref{thm:regretbound} to a finite collection of approximating classes, and then in Corollary \ref{cor:knownclass} establish a bound for the PWM rule which shows that it behaves as if we knew the correct class $\mathcal{G'}$ to use ex-ante, in the special case where the optimal allocation in $\mathcal{G}$ is contained in $\mathcal{G}'$. 

We impose the following assumption on our sequence of classes, which we call a \emph{sieve} of $\mathcal{G}$:

\begin{assumption}\label{ass:seqVCdim}
The sequence of classes 
$$\mathcal{G}_1 \subseteq \mathcal{G}_2\subseteq \mathcal{G}_3\subseteq \cdots \subseteq \mathcal{G}_k\subseteq \cdots \subseteq \mathcal{G}$$
is such that each class $\mathcal{G}_k$ has VC dimension $V_k$, which is finite.\footnotemark[2] 
\end{assumption} 

\footnotetext[2]{\cite{kt2015} additionally assume that their class $\mathcal{G}$ is countable so as to avoid potential measurability concerns. We instead choose not to address these concerns explicitly, as is done in most of the literature on classification. See \cite{van1996} for a discussion of possible resolutions to this issue.}

We present three examples of sieves $\mathcal{G}_k$ in Examples \ref{ex:aseq2}, \ref{ex:aseq3}, \ref{ex:aseq1} below. Given a sieve $\{\mathcal{G}_k\}_{k}$, let
$$\hat{G}_{n,k} := \arg\max_{G \in \mathcal{G}_k} W_n(G)~,$$
be the EWM rule in the class $\mathcal{G}_k$. Our goal is to select the appropriate class $k^*$ in which to perform EWM. We do this by selecting the class $k^*$ in the following way:  for each class $\mathcal{G}_k$, suppose we had some (potentially data-dependent) measure $C_n(k)$ of the amount of ``overfitting"  that results from using the rule $\hat{G}_{n,k}$ (Assumption \ref{ass:Cnk} specifies our precise conditions on $C_n(k)$). Given such a measure $C_n(k)$, let $\{t_k\}_{k=1}^\infty$ be an increasing sequence of real numbers, and define the following penalized objective function:
\begin{equation}
R_{n,k}(G) := W_n(G) - C_n(k) - \sqrt{\frac{t_k}{n}}~.
\label{0eqn1}
\end{equation}
Then the \emph{penalized welfare maximization} rule $\hat{G}_n$ is defined as follows:
$$\hat{G}_n := \hat{G}_{n,\hat{k}^*}~,$$
where 
$$\hat{k}^* := \arg\max_k R_{n,k}(\hat{G}_{n,k})~.$$ 
In words, the PWM rule selects an allocation which maximizes a penalized version of the empirical welfare, with the penalty for allocations in $\mathcal{G}_k$ given by the term $C_n(k)$ (plus the auxiliary term $\sqrt{t_k/n}$). 

\begin{remark}\label{rem:k/n}
Note that the PWM objective function $R_{n,k}(\cdot)$ includes the term: $\sqrt{t_k/n}$. This component of the objective is a technical device that is used to ensure that the classes get penalized at a sufficiently fast rate as $k$ increases. The dependence of the penalty term on the sequence $\{t_k\}_k$ is somewhat undesirable, as it implies that the size of the penalty term for a given class depends on the specific choice of the sequence $\{t_k\}_k$. This technical device seems$-$however$-$unavoidable, and similar terms are pervasive throughout the literature on model selection in classification: see \cite{kolt2001}, \cite{bartlett2002}, \cite{boucheron2005}, \cite{kolt2008}. Nevertheless, as we will show, our results hold for many choices of $\{t_k\}_{k=1}^\infty$ (including our preferred choice $t_k = k$), and the choice is reflected explicitly in the bounds that we derive. Moreover, if one is only interested in using PWM in settings where the sequence of classes is finite, then we will show that the $\sqrt{t_k/n}$ term is not required. For simplicity, and unless otherwise specified, we will present all of our results with the specific choice $t_k = k$; in practice we find that overall performance of our procedure is essentially unaffected by this decision.
\end{remark}

\begin{remark}\label{rem:detseq}
As noted by \cite{kt2015}, given a sieve $\{\mathcal{G}_k\}_{k}$, one can use their results to derive uniform (w.r.t $\mathcal{P}(M,\kappa)$) bounds on the estimation error. If one has in addition uniform bounds on the approximation bias, then one can construct a decision rule $\hat{G}_{n,k(n)}$, where $k(n)$ minimizes sum of these bounds. However, the merit of such an approach would depend on obtaining ``good" computable bounds for the estimation and approximation error, which may be difficult to do in practice. For instance, the uniform bounds on the estimation error from \cite{kt2015} depend on the VC dimension of the classes $\{\mathcal{G}_k\}_k$ which may be hard to bound precisely. Furthermore, a deterministic choice of $k(n)$ may lead to suboptimal rates if the true DGP satisfies additional regularity conditions which may be unknown to the econometrician. Given these challenges, PWM displays two advantages. First, PWM will perform$-$in a data-driven way$-$the optimal tradeoff between the approximation and estimation error, without relying on explicit bounds for these quantities. Second, PWM will select the subclass $\hat{k}$ over which to perform EWM in a way that adapts to additional ``regularities" that may be satisfied by the true DGP.
\end{remark}

We present three examples of sieves $\{\mathcal{G}_k\}_k$:

\begin{example}\label{ex:aseq2}
Recall the class of threshold allocations in one dimension introduced in Example \ref{ex:treatalloc2}. Now we introduce the class of threshold allocations in $K$ dimensions. Let $x = (x_1, ..., x_k) \in \mathcal{X} \subset \mathbb{R}^K$, and consider the following class $\mathcal{G}$:
\[\mathcal{G} = \left\{G \subset \mathcal{X}: G = \{x \in \mathcal{X}: s_kx_k \le \bar{x}_k \hspace{1mm} \text{for} \hspace{1mm} k \in \{1, ..., K\}\}, \bar{x} \in \mathbb{R}^K, s \in \{-1, 1\}^K\right\}\]
For large $K$, the VC dimension of $\mathcal{G}$ can become large relative to the sample size, and we may want to base treatment only on a smaller subset of the covariates. This is a variant of the best-subset selection problem (see \cite{chen2016} for related results in a classification context). However, the question still remains as to \emph{how many} covariates to consider. Consider the sieve sequence $\{\mathcal{G}_k\}_{k=1}^K$, where $\mathcal{G}_k$ corresponds to the class of threshold allocations that uses $(k-1)$ out of $K$ covariates, then PWM applied to this sieve can determine, in a data-driven way, the number of covariates to use for treatment assignment  In Appendix B.3 we illustrate PWM's ability to reduce regret in this context in a simulation study.\end{example}

\begin{example}\label{ex:aseq3}
Recall the class of monotone allocations introduced in Example \ref{ex:treatalloc3}. Suppose that $\mathcal{X} = [0,1]^2$, so that $\mathcal{G}$ has infinite VC dimension \citep[see][for a proof of this fact]{devroye1996}. We will construct a sieve for $\mathcal{G}$ where we approximate sets in $\mathcal{G}$ with sets that feature monotone, piecewise linear boundaries. We proceed in three steps.

First define, for $T$ an integer and $0 \le j \le T$, the following function $\psi_{T,j}: [0,1] \rightarrow [0,1]$:
\[ \psi_{T,j}(x) = \begin{cases} 
      1 - |Tx - j|, & x \in [\frac{j-1}{T},\frac{j+1}{T}] \cap [0,1] \\
      0, & \text{otherwise}~.
   \end{cases}
\]
The function $\psi_{T,j}(\cdot)$ is simply a triangular kernel whose base shifts with $j$ and is scaled by $T$. Next, using these functions, define the following classes $\mathcal{S}_k$:
$$\mathcal{S}_k = \Big\{G: G = \{x = (x_1,x_2) \in \mathcal{X}| \hspace{1mm} \sum_{j=0}^T \theta_j \psi_{T,j}(x_1) + x_2 \ge 0\} \hspace{2mm}  \text{for} \hspace{2mm} \theta_j \in \mathbb{R}, \hspace{1mm} \forall \hspace{1mm} 0 \le j \le T \Big\}~,$$ 
where $T = 2^{k-1}$. It can be shown using results in \cite{dudley1999} that $\mathcal{S}_k$ has VC dimension $T+2$. In words, the sets in $\mathcal{S}_k$ divide the covariate space into treatment and control such that the resulting boundary is a piecewise linear curve. 

Finally, to construct our approximating class $\mathcal{G}_k$, we modify the class $\mathcal{S}_k$ to ensure that the resulting treatment allocations are monotone. For $T$ an integer, let $D_T$ be the following $T \times (T+1)$ ``difference" matrix:
 \[D_T:=\begin{bmatrix}
    -1 & 1 & 0 & \dots  &0& 0 \\
    0 & -1 & 1 & \dots  &0& 0 \\
    \vdots & \vdots & \vdots & \ddots & \vdots& \vdots \\
    0 & 0 & 0 & \dots &-1 & 1
\end{bmatrix}~.
\]
Then $\mathcal{G}_k$ is defined as follows:
$$\mathcal{G}_k = \Big\{G: G\in \mathcal{S}_k \hspace{2mm} \text{and} \hspace{2mm} D_T \Theta_T \ge 0~,~\Theta_T = [\theta_0 \cdots \theta_T]'\Big\}~,$$
for $T = 2^{k-1}$. This construction, which we borrow from \cite{beresteanu2004}, is useful as it imposes monotonicity through a \emph{linear} constraint, which is ideal for our implementation. In Section 4, we use this sequence of approximating classes in an application to the JTPA study, and then derive bounds on the maximum regret of PWM when $\mathcal{X} = [0,1]^2$; Proposition \ref{prop:monbias} provides a uniform rate at which $W^*_{\mathcal{G}_k} \rightarrow W^*_{\mathcal{G}}$ under some additional regularity conditions, and Corollary \ref{cor:monrate} derives the corresponding bound on maximum $\mathcal{G}$-regret.
\end{example}

\begin{example}\label{ex:aseq1}
Decision-tree based policy classes have recently become popular in treatment choice \citep[see for example][]{kallus2016, athey2017}.  Suppose the planner faces no restrictions on treatment assignment, so that $\mathcal{G}$ is the class of all measurable subsets of $\mathcal{X}$. In this case we could consider approximating $\mathcal{G}$ via decision trees of \emph{increasing depth}. PWM could them be used to select the appropriate depth to use in practice.
\end{example}

We are now prepared to state the main results of the paper. We require the following high-level condition on the penalty $C_n(k)$:
\begin{assumption}\label{ass:Cnk}There exist positive constants $c_0$ and $c_1$ such that $C_{n}(k)$ satisfies the following tail inequality for every $n$, $k$, and for every $\epsilon > 0$:
$$\sup_{P \in \mathcal{P}(M,\kappa)} P^n(W_n(\hat{G}_{n,k}) - W(\hat{G}_{n,k}) - C_n(k) > \epsilon) \le c_1e^{-2c_0n\epsilon^2}~.$$
\end{assumption}

Let us provide some intuition for this assumption. Given an EWM rule $\hat{G}_{n,k}$, the value of the empirical welfare is given by $W_n(\hat{G}_{n,k})$. To evaluate $\mathcal{G}$-regret, we would ideally like to know the value of population welfare $W(\hat{G}_{n,k})$. Although the latter quantity is unknown, if we could define the (infeasible) penalty $C_n(k)$ as $W_n(\hat{G}_{n,k}) - W(\hat{G}_{n,k})$, then the penalized objective $W_n(\hat{G}_{n,k}) - C_n(k)$ would be exactly equal to $W(\hat{G}_{n,k})$. Since implementing such a $C_n(k)$ is impossible, our assumption requires that our feasible penalty be a good (empirical) upper bound on $W_n(\hat{G}_{n,k}) - W(\hat{G}_{n,k})$.  In Section \ref{ss:PWMpen}, we provide some specific examples of penalties that satisfy this assumption. In particular, we show that a standard ``hold-out" procedure can be formalized as such a penalty. We are now ready to state our main workhorse result: an \emph{oracle inequality} that characterizes the $\mathcal{G}$-regret of the PWM rule.

\begin{theorem}\label{thm:regretbound}
Suppose that Assumptions \ref{ass:so}, \ref{ass:BO}, \ref{ass:seqVCdim} and \ref{ass:Cnk} hold, and set $t_k = k$ in (\ref{0eqn1}). Then there exist constants $\Delta$ and $c_0$ such that for every $P \in \mathcal{P}(M,\kappa)$:
$$E_{P^n}[W^*_\mathcal{G} - W(\hat{G}_n)] \le \inf_k \Big[E_{P^n}[C_n(k)]+ \big(W^*_{\mathcal{G}} - W^*_{\mathcal{G}_{k}}\big) + \sqrt{\frac{k}{n}}\Big]+ \sqrt{\frac{\log(\Delta e)}{2c_0n}}~.$$
\end{theorem}
Theorem \ref{thm:regretbound} forms the basis of all the results we present in Sections \ref{ss:PWMres} and \ref{ss:PWMpen}. It says that, at least from the perspective of \emph{pointwise} (as opposed to maximum) $\mathcal{G}$-regret, the PWM rule is able to balance the tradeoff between $E_{P^n}[C_n(k)]$ and the approximation error, at the cost of adding two additional terms that are $O(1/\sqrt{n})$. The relative importance of these terms is hard to quantify at this level of generality, and we will attempt to shed some light on them, for specific penalties, in Section \ref{ss:PWMpen}. Note that this result does not \emph{quite} accomplish our initial goal of balancing the estimation and approximation error along our sieve: it is possible to choose a $C_n(k)$ that satisfies Assumption \ref{ass:Cnk} for which $E_{P^n}[C_n(k)]$ is too large a bound for the estimation error. For this reason, we also impose the requirement that any penalty we consider should have the following additional property:

\begin{assumption}\label{ass:Cnk2} There exists a positive constant $C_1$ such that, for every $n$, $C_{n}(k)$ satisfies
$$\sup_{P \in \mathcal{P}(M,\kappa)} E_{P^n}[C_n(k)] \le C_1\sqrt{\frac{V_k}{n}}~,$$
where $V_k$ is the $VC$ dimension of $\mathcal{G}_k$.
\end{assumption}
This assumption ensures that $E_{P^n}[C_n(k)]$ is comparable to the estimation error for EWM derived in (\ref{eq:EWMregretbound}), which was shown to be rate-optimal (for the class $\mathcal{P}(M,\kappa)$) in (\ref{eq:EWMlowerbound}).

The next result we present is a bound on maximum regret for our first setting of interest: choosing the appropriate approximating class when $\mathcal{G}$ has infinite VC dimension. As discussed in Remark \ref{rem:infVC}, a bound on maximum regret may not exist unless we impose additional regularity conditions on the family of DGPs under consideration. Hence we make the additional assumption that we restrict ourselves to a set of distributions $\mathcal{P}_r$ for which there exists a uniform bound on the approximation error. Note however that we do not assume that the rate of decay of the approximation bias is necessarily known to the econometrician, thus illustrating the ``oracle" nature of our results.

\begin{assumption}\label{ass:unifrate}
Let $\mathcal{P}_r$ be a set of distributions such that
$$\sup_{P \in \mathcal{P}_r} W^*_{\mathcal{G}} - W^*_{\mathcal{G}_{k}} = O(\gamma_k)~,$$
$$\sup_{P \in \mathcal{P}_r \cap \mathcal{P}(M,\kappa)} E_{P^n}[C_n(k)] = O(\zeta(k,n))~,$$
for a sequence $\gamma_k \rightarrow 0$, and $\zeta(k,n)$ non-decreasing in $k$, $\zeta(k,n) \rightarrow 0$ as $n \rightarrow \infty$.
\end{assumption}
The first assumption asserts that we have a uniform bound on the approximation error. We present an example of such a uniform bound for our application in Section \ref{sec:application}.  The second assumption is made to highlight the following possibility: although Assumption \ref{ass:Cnk2} guarantees that we can satisfy this restriction with $\zeta(k,n) = \sqrt{V_k/n}$, it is possible that, once we have imposed that $P$ must lie in $\mathcal{P}_r$, an even \emph{tighter} bound may exist on $C_n(k)$ (see for example the discussion which follows Corollary \ref{cor:monrate} below). We emphasize that PWM will balance the tradeoff between the estimation and approximation error according to the \emph{tighest} possible bounds on $E_P[C_n(k)]$ and $W^*_{\mathcal{G}} - W^*_{\mathcal{G}_{k}}$, regardless of whether or not we know these bounds in a given application.

\begin{remark}\label{rem:margin}
A well-known restriction on the class of distributions which may lead to faster rates $\zeta(k,n)$ for certain choices of $C_n(k)$ is the \emph{margin assumption} \citep[see][for a formal definition in the context of treatment choice]{kt2015}. Roughly, the margin assumption imposes restrictions on the behavior of $\tau(\cdot)$ near zero, and thus allows for faster than root-n rates of convergence. Although the study of margin-adaptive penalties is beyond the scope of our paper, \cite{massart2007}  argues (in a classification context) that the hold-out penalty is margin-adaptive. We introduce this penalty in Section \ref{ss:PWMpen} below.
\end{remark}

Given Assumption \ref{ass:unifrate}, we immediately obtain our first corollary:

\begin{corollary}\label{cor:regretrate} Under Assumptions \ref{ass:so}, \ref{ass:BO}, \ref{ass:seqVCdim}, \ref{ass:Cnk}, and \ref{ass:unifrate}, we have that $$\sup_{P \in \mathcal{P}_r \cap \mathcal{P}(M,\kappa)} E_{P^n}[W^*_\mathcal{G} - W(\hat{G}_n)] \le \inf_k \Big[O(\zeta(k,n)) + O(\gamma_k) + \sqrt{\frac{k}{n}}\Big]+ \sqrt{\frac{\log(\Delta e)}{2c_0n}}~.$$
\end{corollary}

 As mentioned in Remark \ref{rem:detseq}, if $\{\zeta(k,n)\}_{k,n}$ and $\{\gamma_k\}_k$ were known, then we could achieve such a result with a deterministic sequence $k(n)$. The strength of the PWM rule then is that we achieve the \emph{same} behavior for any class $\mathcal{G}$ and approximating sequence $\{\mathcal{G}_k\}_k$ without having to know these quantities in practice. We present an application of this result in Section \ref{sec:application}, in the setting of Example \ref{ex:aseq3}.

The second Corollary we present specializes Theorem \ref{thm:regretbound} to our second setting of interest: the appropriate selection of a subclass when the VC-dimension of $\mathcal{G}$ is finite and large relative to the sample size. The result highlights two important points. First, it shows that by balancing the trade-off between the approximation and estimation error, PWM can potentially lead to a reduction in regret (relative to EWM) for values of the sample size that are comparable in magnitude to the VC-dimension of $\mathcal{G}$. Second, it illustrates how our bound changes when the sieve is finite and we drop the auxiliary $\sqrt{k/n}$ component of our penalty.

\begin{corollary}\label{cor:regretpt} Suppose that Assumptions \ref{ass:so}, \ref{ass:BO}, \ref{ass:seqVCdim}, \ref{ass:Cnk}, and \ref{ass:Cnk2} hold, and that $\mathcal{G}_K=\mathcal{G}$ for some finite $K$. Furthemore, suppose that in our definition of the penalty we omit the term $\sqrt{k/n}$. Then we have that
$$E_{P^n}[W^*_\mathcal{G} - W(\hat{G}_n)] \le \inf_{1\leq k \leq K} \Big[C_1\sqrt{\frac{V_k}{n}}+ \big(W^*_{\mathcal{G}} - W^*_{\mathcal{G}_{k}}\big)\Big]+ \sqrt{\frac{\log(Kc_1e)}{2c_0n}}~.$$
\end{corollary}

Note that if the above bound is minimized at $k = K$, then the approximation error $W^*_{\mathcal{G}} - W^*_{\mathcal{G}_{k}}$ is zero and the resulting bound is comparable to the one derived in (\ref{eq:EWMregretbound}), with one additional term. In Section \ref{ss:PWMpen} we argue that for specific choices of $C_n(k)$ this term can be quantified more precisely.

Our final corollary of Section \ref{ss:PWMres} considers the particular setting in which the constrained optimum $W^*_\mathcal{G}$ over the class $\mathcal{G}$ is \emph{achieved} in $\mathcal{G}_{k_0}$, for some $k_0$, but that this class is unknown to the econometrician. The result shows that the resulting upper bound on maximum regret for PWM is as if we had performed EWM in the appropriate class $\mathcal{G}_{k_0}$.

\begin{corollary}\label{cor:knownclass} Suppose that Assumptions \ref{ass:so}, \ref{ass:BO}, \ref{ass:seqVCdim}, \ref{ass:Cnk}, and \ref{ass:Cnk2} hold, and let $\mathcal{P}_k \subset \mathcal{P}(M,\kappa)$ be the set of distributions such that $G^* \in \mathcal{G}_{k}$, then
$$\sup_{P \in \mathcal{P}_k} E_{P^n}[W^*_\mathcal{G} - W(\hat{G}_n)] \le C_1\sqrt{\frac{V_k}{n}} + \sqrt{\frac{k}{n}} + \sqrt{\frac{\log(\Delta e)}{2c_0n}}~.$$
Furthermore, if $\{\mathcal{G}_k\}_{k=1}^K$ is finite, and we do not include the $\sqrt{k/n}$ term as discussed in Remark \ref{rem:k/n}, then we have that:
$$\sup_{P \in \mathcal{P}_k}E_{P^n}[W^*_\mathcal{G} - W(\hat{G}_n)]  \le C_1\sqrt{\frac{V_{k}}{n}} + \sqrt{\frac{\log(Kc_1e)}{2c_0n}}~,$$
where $c_0$, $c_1$ are as in Assumption \ref{ass:Cnk}.
\end{corollary}

\subsection{Penalized Welfare Maximization: Some Examples of Penalties}\label{ss:PWMpen}
This section illustrates the results of Section \ref{ss:PWMres} with two concrete choices for the penalty $C_n(k)$, and quantifies the size of the auxiliary term in the bound of Theorem \ref{thm:regretbound} for these penalties. The first penalty we present, the holdout penalty, formalizes a sample-splitting procedure. The second penalty, the Rademacher penalty, does not involve sample-splitting but could be computationally burdensome in practice. Both of the penalties share the property that they do not require that the practitioner have precise bounds on the VC dimensions $V_k$ of the approximating classes (only that they are finite), which we feel is important to make the method broadly applicable. 

\subsubsection{The Holdout Penalty}
The first penalty we introduce is motivated by the following idea: fix some number $\ell \in (0,1)$ such that $m := n(1-\ell)$ (for expositional simplicity suppose that $m$ is an integer), and let $r := n - m$.
Given our original sample $S_n = \{(Y_i,D_i,X_i)\}_{i=1}^n$, let $S_n^E := \{(Y_i,D_i,X_i)\}_{i=1}^m$ denote what we call the \emph{estimating} sample, and let $S_n^T := \{(Y_i,D_i,X_i)\}_{i=m+1}^n$ denote the \emph{testing} sample. Now, using $S_n^E$, compute $\hat{G}_{m,k}$ for each $k$. Intuitively, we could get a sense of the efficacy of $\hat{G}_{m,k}$ by applying this rule to the subsample $S_n^T$ and computing the empirical welfare $W_{r}(\hat{G}_{m,k})$. We could then select the class $k$ that results in the highest empirical welfare $W_{r}(\hat{G}_{m,k})$.

This idea can be formalized in our framework by treating it as a PWM-rule on the estimating sample, with the following penalty: for each EWM rule $\hat{G}_{m,k}$ estimated on $S_n^E$, let $$W_m(\hat{G}_{m,k}) = \frac{1}{m}\sum_{i=1}^m\tau_i {\bf 1}\{X_i \in \hat{G}_{m,k}\}~,$$
be the empirical welfare of the rule $\hat{G}_{m,k}$ on  $S_n^E$ and let
$$W_r(\hat{G}_{m,k}) = \frac{1}{r}\sum_{i=m+1}^n\tau_i {\bf 1}\{X_i \in \hat{G}_{m,k}\}~,$$
be the empirical welfare of the rule $\hat{G}_{m,k}$ on  $S_n^T$. We define the \emph{holdout} penalty to be
$$C_m(k) := W_m(\hat{G}_{m,k}) - W_r(\hat{G}_{m,k})~.$$

Now, recall that the PWM rule is given by
$$\hat{G}_m = \arg\max_k \left[W_m(\hat{G}_{m,k}) - C_m(k) - \sqrt{\frac{k}{m}}\right]~,$$
which, given the definition of $C_m(k)$, simplifies to
$$\hat{G}_m = \arg\max_k \left[W_r(\hat{G}_{m,k}) - \sqrt{\frac{k}{m}}\right]~.$$
Hence we see that the PWM rule with the holdout penalty reproduces the intuition presented above (with the usual addition of the $\sqrt{k/m}$ term; see Remark \ref{rem:k/n}).

We check the conditions of Assumptions \ref{ass:Cnk} and \ref{ass:Cnk2}:
\begin{lemma}\label{lem:holdout}
Consider Assumptions \ref{ass:so}, \ref{ass:BO}, \ref{ass:seqVCdim}. Suppose we have a sample of size $n$ and recall that $m = n(1-\ell)$ and $r = n - m$. Let $C_m(k)$ be the holdout penalty as defined above. Then we have that 
$$P^n(W_m(\hat{G}_{m,k}) - W(\hat{G}_{m,k}) - C_m(k) > \epsilon) \le \exp\Big(-2\Big(\frac{\kappa}{M}\Big)^2n\ell\epsilon^2\Big),$$
and $$E_{P^n}[C_m(k)] \le C\frac{M}{\kappa\sqrt{(1-\ell)}}\sqrt{\frac{V_k}{n}}~,$$
where $C$ is the same universal constant that appears in equation (\ref{eq:EWMregretbound}).
\end{lemma}

With Lemma \ref{lem:holdout} established, Theorem \ref{thm:regretbound} becomes:

\begin{proposition}\label{prop:holdoutbound}
Consider Assumptions \ref{ass:so}, \ref{ass:BO}, \ref{ass:seqVCdim}. Suppose we have a sample of size $n$, and let $m = n(1-\ell)$, $r = n - m$. Let $C_m(k)$ be the holdout penalty as defined above. Then we have that for every $P \in \mathcal{P}(M,\kappa)$:
$$E_{P^n}[W^*_\mathcal{G} - W(\hat{G}_m)] \le \inf_k \Big[E_{P^n}[C_n(k)]+ \big(W^*_{\mathcal{G}} - W^*_{\mathcal{G}_{k}}\big) + \sqrt{\frac{k}{n}}\Big]+ g(M,\kappa,\ell)\frac{M}{\kappa\sqrt{\ell}}\sqrt{\frac{1}{n}}~,$$
with $$E_{P^n}[C_n(k)] \le C\frac{M}{\kappa\sqrt{(1-\ell)}}\sqrt{\frac{V_k}{n}}~,$$ where $C$ is the same universal constant as that in equation (\ref{eq:EWMregretbound}) and 
$$g(M,\kappa,\ell) := 2\sqrt{\log\Big(\sqrt{\frac{e}{2\ell}}\frac{M}{\kappa}\Big)}~.$$
\end{proposition}

As we show in the next section, the bound in Proposition \ref{prop:holdoutbound} is similar to what we derive for the Rademacher penalty, but with larger constants which reflect the fact that we split the sample. However, a major benefit of the holdout penalty lies in the fact that it is simple to implement. The only remaining issue is how to split the data. Although we do not study this problem formally, we have found that it is more important to focus on accurate estimation of the rule $\hat{G}_{m,k}$ than on the computation of $W_r(\hat{G}_{m,k})$. In other words, we recommend that the estimating sample $S_n^E$ be a large proportion of the original sample $S_n$. Throughout the rest of the paper we designate three quarters of the sample as the estimating sample.

\subsubsection{The Rademacher Penalty}
The second penalty we present is attractive in that it does not introduce sample splitting, but may be computationally burdensome when compared to the holdout procedure. Let $S_n := \{(Y_i,D_i,X_i)\}_{i=1}^n$ be the observed data. Then the \emph{Rademacher} penalty is given by
$$C_n(k) = E_{\sigma}\Big[\sup_{G \in \mathcal{G}_k} \frac{2}{n}\sum_{i=1}^n \sigma_i \tau_i {\bf 1}\{X_i \in G\}\mid S_n\Big]~,$$
where $\tau_i$ is defined as in equation (\ref{eq:EWMobjective}), and $\{\sigma_1,...,\sigma_n\}$ are a sequence of i.i.d Rademacher variables, i.e. they take on the values $\{-1,1\}$, each with probability half.

To clarify the origin of this penalty, recall that $C_n(k)$ must be a good upper bound on $W_n(\hat{G}_{n,k}) - W(\hat{G}_{n,k})$, which is the requirement of Assumption \ref{ass:Cnk}. Bounding such quantities is common in the study of empirical processes, and the usual first step is to use what is known as \emph{symmetrization}, which gives the following bound:
$$E_{P^n}[\sup_{G\in\mathcal{G}} W_n(G) - W(G)] \le E_{P^n}\Big[E_{\sigma}\big[\sup_{G\in\mathcal{G}}\frac{2}{n}\sum_{i=1}^n \sigma_i \tau_i {\bf 1}\{X_i \in G\}\mid S_n\big]\Big]~.$$
It is thus this inequality that inspires the definition of $C_n(k)$. The concept of Rademacher complexity\footnotemark[4] is pervasive throughout the statistical learning literature \citep[see for example][]{kolt2001, bartlett2002rad, bartlett2002}. Intuitively, it measures a notion of complexity that is finer than that of VC dimension, and is at the same time computable from the data at hand. 

\footnotetext[4]{Note that the definition of Rademacher complexity is slightly different than the definition of our penalty. Here we follow \cite{bartlett2002} and do not include the absolute value in our definition of the penalty.}

First we prove that the conditions of Assumptions \ref{ass:Cnk} and \ref{ass:Cnk2} hold for the Rademacher penalty:

\begin{lemma}\label{lem:rademacher}
Consider Assumptions \ref{ass:so}, \ref{ass:BO}, \ref{ass:seqVCdim}. Let $C_n(k)$ be the Rademacher penalty as defined above. Then we have that 
$$P^n(W_n(\hat{G}_{n,k}) - W(\hat{G}_{n,k}) - C_n(k) > \epsilon) \le \exp\Big(-2\Big(\frac{\kappa}{3M}\Big)^2n\epsilon^2\Big),$$
and $$E_{P^n}[C_n(k)] \le C\frac{M}{\kappa}\sqrt{\frac{V_k}{n}}~,$$
where $C$ is the same universal constant that appears in equation (\ref{eq:EWMregretbound}).
\end{lemma}

We are thus able to refine Theorem \ref{thm:regretbound} to the case of the Rademacher penalty.

\begin{proposition}\label{prop:rademacherbound}
Consider Assumptions \ref{ass:so}, \ref{ass:BO}, \ref{ass:seqVCdim}. Let $C_n(k)$ be the Rademacher penalty as defined above. Then we have that for every $P \in \mathcal{P}(M,\kappa)$:
$$E_{P^n}[W^*_\mathcal{G} - W(\hat{G}_n)] \le \inf_k \Big[E_{P^n}[C_n(k)] + \big(W^*_{\mathcal{G}} - W^*_{\mathcal{G}_{k}}\big) + \sqrt{\frac{k}{n}}\Big]+ g(M,\kappa)\frac{M}{\kappa}\sqrt{\frac{1}{n}}~,$$
with $E_{P^n}[C_n(k)] \le C\frac{M}{\kappa}\sqrt{\frac{V_k}{n}}$, where $C$ is the same universal constant as that in equation (\ref{eq:EWMregretbound}) and 
$$g(M,\kappa) := 6\sqrt{\log\Big(\frac{3\sqrt{e}}{\sqrt{2}}\frac{M}{\kappa}\Big)}~.$$
\end{proposition}

\begin{remark}\label{rem:tt1}
In Appendix B we perform a back-of-the-envelope calculation that provides insight into the size of $g(M,\kappa)$, and compares it to the size of the universal constant $C$ derived in \cite{kt2015}. \end{remark}

\subsection{Penalized Welfare Maximization: Estimated Propensity Score}\label{ss:PWMe}

In this section we present a modification of the PWM rule where the propensity score is not known and must be estimated from the data. This situation would arise if the planner had access to observational data instead of data from a randomized experiment. Before describing our modification of the PWM rule, we first review results about the corresponding modification of the EWM rule. The modification we consider here is what \cite{kt2015} call the \emph{e-hybrid} EWM rule. Recall the EWM objective function as defined in equation (\ref{eq:EWMobjective}). To define the e-hybrid EWM rule we modify this objective function by replacing $\tau_i$ with
$$\hat{\tau}_i := \Big[\frac{Y_iD_i}{\hat{e}(X_i)} - \frac{Y_i(1-D_i)}{1-\hat{e}(X_i)}\Big]{\bf 1}\{\epsilon_n\le\hat{e}(X_i)\le1-\epsilon_n\}~,$$
where $\hat{e}(\cdot)$ is an estimator of the propensity score, and $\epsilon_n$ is a trimming parameter such that $\epsilon_n = O(n^{-\alpha})$ for some $\alpha > 0$. The e-hybrid EWM objective function is defined as follows:
$$W^e_n(G):=\frac{1}{n}\sum_{i=1}^n\hat{\tau}_i{\bf 1}\{X_i \in G\}~.$$

Since we are now estimating the propensity score, we must impose additional regularity conditions on $P$ to guarantee a uniform rate of convergence. We impose the following high level assumption:

\begin{assumption}\label{ass:ereg}
Given an estimator $\hat{e}(\cdot)$, let $\mathcal{P}_e$ be a class of data generating processes such that
$$\sup_{P\in\mathcal{P}_e} E_{P^n}\Big[\frac{1}{n}\sum_{i=1}^n|\hat{\tau}_i - \tau_i|\Big] = O(\phi_n^{-1})~,$$
where $\phi_n \rightarrow \infty$. 
\end{assumption}

Although we do not explore low-level conditions that satisfy this assumption here, \cite{kt2015} do so in their paper. Let $\hat{G}_{e-hybrid}$ be the solution to the e-hybrid problem in a class $\mathcal{G}$ of finite VC dimension, then \cite{kt2015} derive the following bound on maximum $\mathcal{G}$-regret:

%\begin{equation} \label{eq:hybEWMobjective}
\[\sup_{P \in \mathcal{P}_e \cap \mathcal{P}(M,\kappa)} E_{P^n}\Big[W^*_{\mathcal{G}} - W(\hat{G}_{e-hybrid})\Big] \le O(\phi_n^{-1} \vee n^{-1/2})~.\]
%\end{equation}
With a non-parametric estimator of $e(\cdot)$, $\phi_n$ will generally be slower than $\sqrt{n}$ and hence determine the rate of convergence. In a recent paper, \cite{athey2017} argue that more sophisticated estimators of the welfare objective can improve performance relative to the e-hybrid rule, and derive corresponding bounds on the maximum regret of their procedure. Importantly, by exploiting an orthogonal moments construction, the procedure in \cite{athey2017} converges at a $\sqrt{n}$-rate even when the propensity score is estimated non-parametrically. Modifying our method using their techniques would be an interesting direction for future work. 

We now present the construction of the corresponding e-hybrid PWM estimator. Let $\mathcal{G}$ be an arbitrary class of allocations, and let $\{\mathcal{G}_k\}_k$ be some approximating sequence for $\mathcal{G}$. Let $\hat{G}^e_{n,k}$ be the hybrid EWM rule in the class $\mathcal{G}_k$. Let $C^e_n(k)$ be our penalty for the hybrid PWM rule. We require that the penalty satisfies the following properties:

\begin{assumption}\label{ass:Cnke}(Assumptions on $C^e_n(k)$)\newline
In addition to making assumptions about $C^e_n(k)$, we assume there exists an ``infeasible penalty" $\tilde{C}_n(k)$ with the following properties:
\begin{itemize}
\item There exist positive constants $c_0$ and $c_1$ such that $\tilde{C}_n(k)$ satisfies the following tail inequality for every $n$, $k$ and for every $\epsilon > 0$:
$$\sup_{P \in \mathcal{P}_e\cap\mathcal{P}(M,\kappa)} P^n(W_n(\hat{G}^e_{n,k}) - W(\hat{G}^e_{n,k}) - \tilde{C}_{n}(k) > \epsilon) \le c_1e^{-2c_0n\epsilon^2}$$
\item There exists a positive constant $C_1$ such that, for every $n$, $\tilde{C}_n(k)$ satisfies
$$\sup_{P \in \mathcal{P}_e\cap\mathcal{P}(M,\kappa)} E_{P^n}[\tilde{C}_n(k)] \le C_1\sqrt{\frac{V_k}{n}}~,$$
where $V_k$ is the VC dimension of $\mathcal{G}_k$.
\item $\tilde{C}_n(k)$ and $C_n^e(k)$ are such that
$$\sup_{P \in \mathcal{P}_e\cap\mathcal{P}(M,\kappa)} E_{P^n}\left[\sup_k\left|C_n^e(k) - \tilde{C}_n(k)\right|\right] = O(\phi^{-1}_n)~.$$
\end{itemize}
\end{assumption}

Note that we have introduced an object $\tilde{C}_n(k)$ which we call an \emph{infeasible penalty}. The first bullet point asserts that the infeasible penalty obeys a similar tail inequality to $C_n(k)$, except that $\tilde{C}_n(k)$ satisfies this assumption with respect to the \emph{e-hybrid} EWM rule. However, we evaluate the hybrid rule through the empirical objective $W_n(\cdot)$, \emph{which is the objective when the propensity score is known}. This is our motivation for calling $\tilde{C}_n(k)$ an infeasible penalty. Luckily, $\tilde{C}_n(k)$ is purely a theoretical device and does not serve a role in the actual implementation of PWM. We provide an example of such an infeasible penalty in the setting of the holdout penalty below.
The second bullet point is the same as Assumption \ref{ass:Cnk2}, but now with respect to the infeasible penalty $\tilde{C}_n(k)$. The third bullet simply links the true penalty $C_n^e(k)$ to the infeasible penalty $\tilde{C}_n(k)$ in such a way that both should agree asymptotically and do so at an appropriate rate.

Given this, we obtain the following analogue to Theorem \ref{thm:regretbound}:
\begin{theorem}\label{thm:eregretbound}
Given assumptions \ref{ass:so}, \ref{ass:BO}, \ref{ass:seqVCdim}, \ref{ass:ereg} and \ref{ass:Cnke}, there exist constants $\Delta$ and $c_0$ such that for every $P \in \mathcal{P}_e \cap \mathcal{P}(M,\kappa)$:
$$E_{P^n}[W^*_\mathcal{G} - W(\hat{G}^e_n)] \le \inf_k \Big[E_{P^n}[\tilde{C}_n(k)] + \big(W^*_{\mathcal{G}} - W^*_{\mathcal{G}_{k}}\big) + \sqrt{\frac{k}{n}}\Big]+ O(\phi^{-1}_n) + \sqrt{\frac{\log(\Delta e)}{2c_0n}}~.$$
\end{theorem}

As we can see, the only difference between this bound and the bound derived in Theorem \ref{thm:regretbound} is that there is an additional term of order $\phi^{-1}_n$. 

Next, we check that the conditions in Assumption \ref{ass:Cnke} are satisfied with a modified version of the holdout penalty (the results for the Rademacher penalty follow similarly). To define the \emph{hybrid holdout} penalty, let $\hat{e}^E(\cdot)$ be the propensity estimated on $S_n^E$, and let $\hat{e}^T(\cdot)$ be the propensity estimated on $S_n^T$. Define
$$W_m^e(G) := \frac{1}{m}\sum_{i=1}^m\hat{\tau_i}^E{\bf 1}\{X_i \in G\}~,$$
where 
$$\hat{\tau_i}^E = \Big[\frac{Y_iD_i}{\hat{e}^E(X_i)} - \frac{Y_i(1-D_i)}{1-\hat{e}^E(X_i)}\Big]{\bf 1}\{\epsilon_n\le\hat{e}^E(X_i)\le1-\epsilon_n\}~.$$
Define $W_r^e(G)$ on the testing sample analogously. Letting $\hat{G}^e_{m,k}$ be the hybrid EWM rule computed on the estimating sample in the class $\mathcal{G}_k$, the hybrid holdout penalty is defined as:
$$C_m^e(k) := W_m^e(\hat{G}^e_{m,k}) - W_r^e(\hat{G}^e_{m,k})~.$$

We must also assert the existence of an infeasible penalty $\tilde{C}_m(k)$ that satisfies our assumptions. The infeasible penalty we consider is given by
$$\tilde{C}_m(k) := W_m(\hat{G}^e_{m,k}) - W_r(\hat{G}^e_{m,k})~,$$
where $W_m(\cdot)$ and $W_r(\cdot)$ are defined as in Section \ref{ss:PWMpen}, that is, they are computed \emph{as if the propensity score were known}. Lemma \ref{lem:holdout} verifies Assumption \ref{ass:Cnke} for the hybrid holdout penalty:

\begin{lemma}\label{lem:eholdout}
Assume Assumptions \ref{ass:so}, \ref{ass:BO}, \ref{ass:seqVCdim}, and \ref{ass:ereg}. Suppose we have a sample of size $n$ and recall that $m = n(1-\ell)$ and $r = n - m$. Let $C^e_m(k)$ be the hybrid holdout penalty and $\tilde{C}_m(k)$ be the infeasible penalty as defined above. Then we have that 
$$P^n(W_m(\hat{G}^e_{m,k}) - W(\hat{G}^e_{m,k}) - \tilde{C}_m(k) > \epsilon) \le \exp\Big(-2\Big(\frac{\kappa}{M}\Big)^2n\ell\epsilon^2\Big),$$
$$E_{P^n}[\tilde{C}_m(k)] \le C\frac{M}{\kappa\sqrt{(1-\ell)}}\sqrt{\frac{V_k}{n}}~,$$
and $$\sup_{P \in \mathcal{P}_e}E_{P^n}\big[\sup_k|C^e_m(k) - \tilde{C}_m(k)|\big] = O(\phi^{-1}_n)~,$$
where $C$ is the same universal constant as that in equation (\ref{eq:EWMregretbound}).
\end{lemma}
We thus obtain an analogous result to Proposition \ref{prop:holdoutbound} for PWM with the hybrid holdout penalty.

\section{An Application using Monotone Allocations}\label{sec:application}
In this section we apply the PWM rule to the sieve we constructed in Example \ref{ex:aseq3} for monotone allocations. First, we apply our method to experimental data from the Job Training Partnership Act (JTPA) Study. Then, we derive bounds on maximum regret in a setting where our class has infinite VC dimension.  

The JTPA study was a randomized controlled trial whose purpose was to measure the benefits and costs of employment and training programs. The study randomized whether applicants would be eligible to receive a collection of services provided by the JTPA related to job training, for a period of 18 months. The study collected background information about the applicants prior to the experiment, as well as data on applicants' earnings for 30 months following assignment \citep[for a detailed description of the study, see][]{bloom1997}.\footnotemark[5] 

\footnotetext[5]{The sample we use is the same as that in \cite{abadie2013}, which we downloaded from \href{https://ideas.repec.org/c/boc/bocode/s457801.html}{ideas.repec.org/c/boc/bocode/s457801.html}. We supplemented this dataset with education data from the \texttt{expbif.dta} dataset available at the W.E. Upjohn Institute website. Observations with years of education coded as `99' were dropped.}

We revisit the setup in \cite{kt2015}, which has frequently been considered in recent related papers. The outcome that we consider is total individual earnings in the 30 months following program assignment. The covariates on which we define our treatment allocations are the individual's years of education and their earnings in the year prior to the assignment. The set of allocations we consider is the set of monotone allocations defined in Example \ref{ex:treatalloc3}, but with a \emph{non-increasing} monotone function. To be precise, let $\mathcal{X}_1$ be the covariate set of years of education, and let $\mathcal{X}_2$ be the covariate set of previous earnings, then the set of allocations we consider is given by:
$$\mathcal{G} = \big\{G : G = \{(x_1,x_2) \in \mathcal{X} | \hspace{1mm} x_2 \le f(x_1) \hspace{1mm} \text{for} \hspace{1mm} f:\mathcal{X}_1 \rightarrow \mathcal{X}_2 \hspace{1mm} \text{non-increasing}\} \big\}~.$$
In the context of this application, this restriction imposes that, the less education you have, the more accessible is the program based on your previous earnings. It is plausible that such a restriction may be exogenously imposed on the planner for political reasons; after all, it may not be politically viable to implement a job-training program where only those with high levels of education or income are accepted. 

The approximating sequence we consider is the one described in Example \ref{ex:aseq3}, but now with a non-increasing monotonicity constraint. Recall that this was a sequence such that the resulting allocations partitioned the covariate space with a progressively refined, piecewise linear, monotone boundary. Given any fixed class in this sequence, we can perform EWM in that class. Figure \ref{fig:EWM1} below illustrates the result of performing EWM on the simplest class in the approximating sequence. This class is equivalent to the class of linear treatment rules from \cite{kt2015}, but with an additional slope constraint. Appendix C describes the computational details of our implementation.

\begin{figure}[!ht]
\includegraphics[scale=0.75]{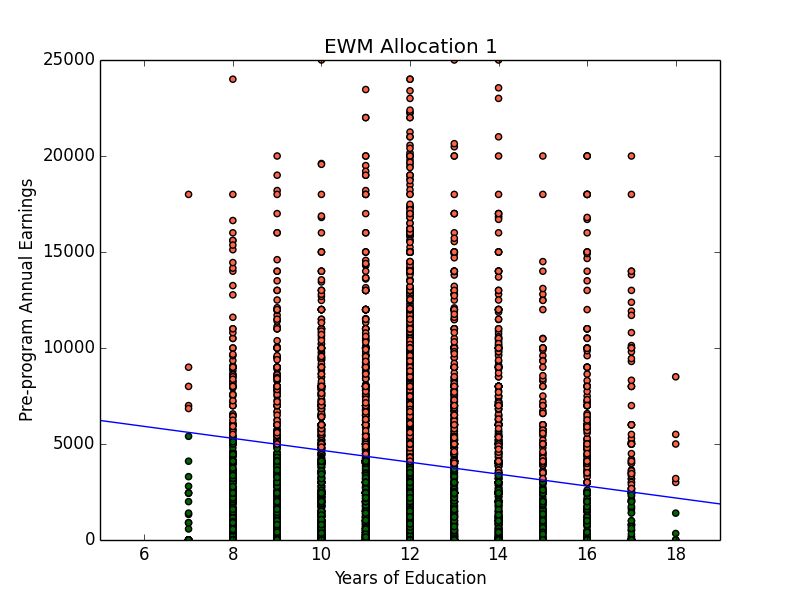} 
\centering
\captionsetup{justification = centering} 
\caption{The resulting treatment allocation from performing EWM in $\mathcal{G}_1$.\\
Each point represents a covariate pair in the sample. The lower left region (dark, green) is the prescribed treatment region, the upper right region (light, red) is the prescribed control region. \label{fig:EWM1}}
\end{figure}

At the other end of the spectrum, we could consider performing EWM in the most complicated class in our approximating sequence: this class corresponds to allocations that stipulate a threshold for previous income at every level of education (note that such a class exists here because years of education is discrete with finite support). Figure \ref{fig:EWM5} below illustrates the result of performing EWM in this class.

\begin{figure}[!ht]
\includegraphics[scale=0.75]{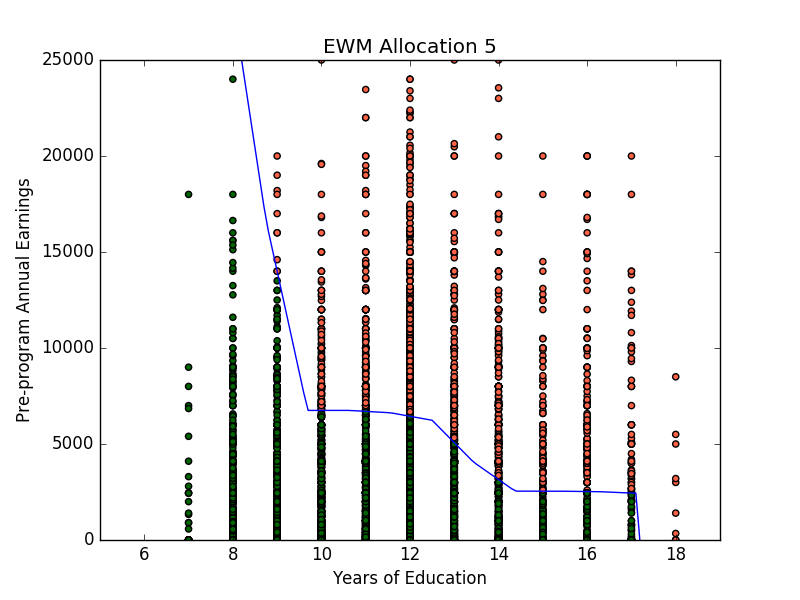} 
\centering
\captionsetup{justification = centering} 
\caption{The resulting treatment allocation from performing EWM in $\mathcal{G}_5$.\\
Each point represents a covariate pair in the sample. The lower left region (dark, green) is the prescribed treatment region, the upper right region (light, red) is the prescribed control region.  \label{fig:EWM5}}
\end{figure}

As we can see, the resulting allocation in the simplest class and in the most complicated class look quite different, and given the option to choose any class from our sequence, it is not obvious which one should be chosen given the size of the experiment. Given that we have a finite sieve in this application, we can view the use of PWM in this context through the lenses of Corollaries \ref{cor:regretpt} or \ref{cor:knownclass}. In Figure \ref{fig:PWM}, we illustrate the result of performing PWM on our sequence of classes, where we used $3/4$ of our sample for estimation. Note that PWM selects the allocation from the second class in our sequence, which corresponds to a piecewise-linear allocation with one possible ``kink". 

\begin{figure}[!ht]
\includegraphics[scale=0.75]{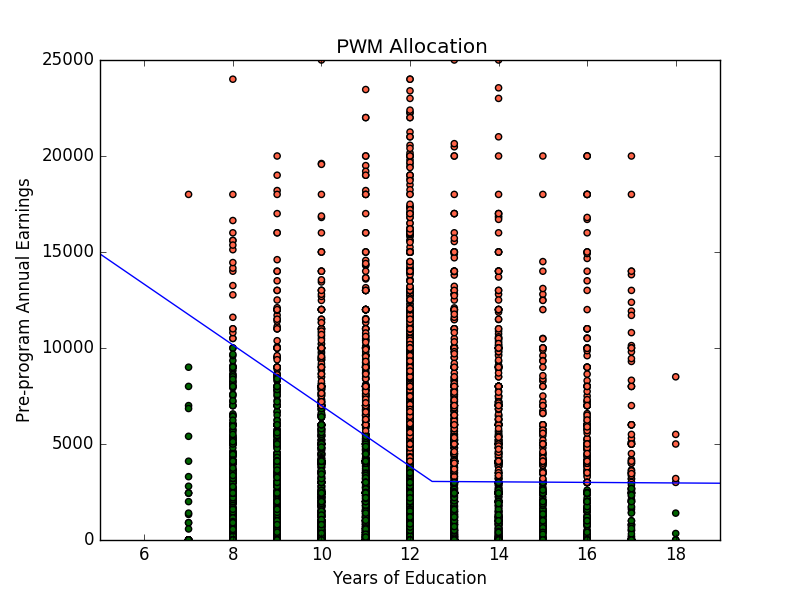} 
\centering
\captionsetup{justification = centering} 
\caption{The resulting treatment allocation from performing PWM on the approximating sequence $\{\mathcal{G}_k\}_{k=1}^5$. 
Each point represents a covariate pair in the sample. The lower left region (dark, green) is the prescribed treatment region, the upper right region (light, red) is the prescribed control region.  \label{fig:PWM}}
\end{figure}

\begin{remark}\label{rem:holdout_test}
In Appendix B we perform a sample splitting exercise to estimate the welfare performance of various decision rules on the JTPA data. In summary, we find that PWM can obtain higher estimated welfare than EWM in this application. However, we emphasize that this difference was not found to be statistically significant. 
\end{remark}

Next, we derive a bound on maximum regret in the setting where $\mathcal{X} = [0,1]^2$, so that our class has infinite VC dimension. We consider the following restriction on the class of distributions:

\begin{assumption}\label{ass:monreg}
Let $\mathcal{P}_r$ be a set of DGPs such that there exists some constant $A > 0$, where for every distribution in $\mathcal{P}_r$, the marginal distribution of $X = (X_1,X_2)$ can be decomposed as follows:
\[P_X(M_1\times M_2) = \int_{M_2}P_{X_1|x_2}(M_1)dP_{X_2}~,\]
where $M_1$ and $M_2$ are measurable subsets of $[0,1]$, and $P_{X_1|x_2}$ is continuous with density bounded above by $A$, for all $x_2 \in [0,1]$.
\end{assumption}

In words, Assumption \ref{ass:monreg} requires that the conditional distribution of $X_1$ given $X_2$ is continuous with a uniformly bounded density. With this regularity condition imposed, we are able to derive the following uniform bound on the approximation bias $W^*_\mathcal{G} - W^*_{\mathcal{G}_k}$:

\begin{proposition}\label{prop:monbias}
Under Assumption \ref{ass:monreg}, the approximation bias of the approximating sequence $\{\mathcal{G}_k\}_{k=1}^\infty$ from Example \ref{ex:aseq3} satisfies
$$\sup_{P \in \mathcal{P}_{r} \cap \mathcal{P}(M,\kappa)} W^*_\mathcal{G} - W^*_{\mathcal{G}_k} \le A\frac{M}{\kappa}2^{-k}~,$$
\end{proposition}

To illustrate the use of Proposition \ref{prop:monbias} in our setting, we derive a bound on maximum regret for monotone allocations. Proposition \ref{prop:monbias} and Corollary \ref{cor:regretrate}, along with the bound on $V_k$ given in Example \ref{ex:aseq3} allow us to conclude that:

\begin{corollary}\label{cor:monrate}
Let $C_n(k)$ be the Rademacher or holdout penalty. Under Assumptions \ref{ass:so}, \ref{ass:BO}, \ref{ass:seqVCdim}, and \ref{ass:monreg}, we have that
$$\sup_{P \in \mathcal{P}_r \cap \mathcal{P}(M,\kappa)} E_{P^n}[W^*_\mathcal{G} - W(\hat{G}_n)] = O\big(n^{-\frac{1}{3}}\big)~.$$
\end{corollary}

Corollary \ref{cor:monrate} establishes a polynomial rate of convergence for PWM. In contrast, we are not aware of any results which would allow us to derive a bound on maximum regret for EWM (or other related methods which do not employ a sieve construction) under Assumption \ref{ass:monreg}. 

In Appendix B, we derive a series of results on the behavior of EWM under suitable entropy restrictions on the class $\mathcal{G}$, which show that under the stronger assumption that $X$ is continuous with a bounded density, EWM in fact achieves a root-n rate (up to a log factor) in this example, and that this rate is optimal. As we explain in Remark B.2, PWM can also achieve the same optimal rate of convergence in this setting, which would \emph{not} be the case for a deterministic $k(n)$ chosen to obtain the rate derived in Corollary \ref{cor:monrate}. This further reinforces the observation made in Remark \ref{rem:detseq} about the adaptation of PWM to additional regularities.

\clearpage

%\begin{small}
\appendix
%
%\end{small}
\clearpage	
%\nocite{*}
\bibliography{references.bib}
\clearpage
\begin{small}
\begin{center}
\LARGE{Supplement to ``Model Selection for Treatment Choice: Penalized Welfare Maximization"}
\end{center}
\section{Proofs of Main Results}\label{sec:appendixA}
Recall that the planner's objective function is given by
\begin{equation}
\label{eqnmax}
W(G)=E_P\left[\left(\frac{YD}{e(X)}-\frac{Y(1-D)}{1-e(X)}\right)\cdot{\bf 1}\{X\in G\}\right]~.
\end{equation}
To each treatment allocation $G \in {\cal G}$ we associate a function $f_G:\mathbb{R}\times{\cal X}\times \{0,1\} \rightarrow \mathbb{R}$ defined by:
$$f_G(Z)=f_G(Y,X,D)=\left(\frac{YD}{e(X)}-\frac{Y(1-D)}{1-e(X)}\right)\cdot{\bf 1}\{X\in G\}~,$$
where $Z = (Y,X,D)$. Let ${\cal F}:=\{f_G:G \in {\cal G}\}$ denote the corresponding set of functions associated to decision rules in ${\cal G}$. By (\ref{eqnmax}), any optimal allocation in ${\cal G}$ solves
$$G^* \in \arg\max_{G \in {\cal G}}E_P\left[\left(\frac{YD}{e(X)}-\frac{Y(1-D)}{1-e(X)}\right)\cdot{\bf 1}\{X\in G\}\right]~.$$
Equivalently, functions associated to optimal allocations solve
$$f^* \in \arg\max_{f \in {\cal F}} E_P f(Z)~.$$
By an abuse of notation, for $G \in {\cal G}$, we set
$$W(f_G)=E_P f_G(Z)~.$$
Given an approximating sequence $\{{\cal G}_k\}_{k}$ of classes of treatment allocations, let $\{{\cal F}_k\}_{k}$ denote the sequence of associated classes of functions.

The following lemma, whose proof is given in \cite{kt2015} (Lemma A.1), establishes the relevant link between the classes of sets $\{\mathcal{G}_k\}_k$ and the classes of functions $\{\mathcal{F}_k\}_k$. It shows that if a class ${\cal G}$ has finite VC dimension, then the associated class ${\cal F}$ is a VC-subgraph class with dimension bounded above by that of ${\cal G}$.
\begin{lemma}\label{lem:VCsbgf}
Let ${\cal G}$ be a VC-class of subsets of ${\cal X}$ with finite VC dimension $V$. Let g be a function from ${\cal Z}:=\mathbb{R}\times{\cal X}\times \{0,1\}$ to $\mathbb{R}$. Then the set of functions ${\cal F}$ defined by
$${\cal F}=\{g(z)\cdot{\bf 1}\{x \in G\} : G \in {\cal G}\}$$
is a VC-subgraph class with dimension at most $V$.
\end{lemma}

For each $k\geq 1$, let $\hat{f}_{n,k}$ be a maximizer of the empirical welfare over the class ${\cal F}_k$; that is:
$$\hat{f}_{n,k} =  \arg\max_{f \in \mathcal{F}_{k}} W_n(f)~,$$
and for $f \in {\cal F}_k$, define the complexity-penalized estimate of welfare by
$$R_{n,k}(f) = W_n(f) - C_n(k) - \sqrt{\frac{k}{n}}~.$$
The PWM rule $\hat{f}_{n,\hat{k}}$ is then chosen such that
$$\hat{k}=\arg\max_{k\geq 1} R_{n,k}(\hat{f}_{n,k})~.$$
In what follows, we set $\hat{f}_n := \hat{f}_{n,\hat{k}}$ and $R_n(\hat{f}_n) := R_{n,\hat{k}}(\hat{f}_{n,\hat{k}})$.

To bound the regret, we decompose it as follows
\begin{equation}
\label{eqna1}
W^*_{\mathcal{F}}-W(\hat{f}_n)=\left(W^*_{\mathcal{F}}-R_n(\hat{f}_n)\right)+\left(R_n(\hat{f}_n)-W(\hat{f}_n)\right).
\end{equation}
The following lemma yields (under Assumption \ref{ass:Cnk}) a subgaussian tail bound for the second term on the right hand side of the preceding equality.
\begin{lemma}\label{lemma:tailbound}
Given Assumption \ref{ass:Cnk}, there exists a positive constant $\Delta$ (that does not depend on n) such that:
$$P(R_n(\hat{f}_n) - W(\hat{f}_n) > \epsilon) \le \Delta e^{-2c_on\epsilon^2}~$$
for every n.
\end{lemma}

\begin{proof}
First note that:
$$P(R_n(\hat{f}_n) - W(\hat{f}_n) > \epsilon) \le P\Big(\sup_k\big(R_{n,k}(\hat{f}_{n,k}) - W(\hat{f}_{n,k})\big) > \epsilon\Big)~,$$
then by the union bound:
$$P\Big(\sup_k\big(R_{n,k}(\hat{f}_{n,k}) - W(\hat{f}_{n,k})\big) > \epsilon\Big) \le \sum_{k} P(R_{n,k}(\hat{f}_{n,k}) - W(\hat{f}_{n,k}) > \epsilon)~.$$
Now by definition of $R_{n,k}$, we have
$$\sum_{k} P(R_{n,k}(\hat{f}_{n,k}) - W(\hat{f}_{n,k}) > \epsilon) = \sum_{k} P\big(W_{n}(\hat{f}_{n,k}) - C_n(k) - W(\hat{f}_{n,k}) > \epsilon + \sqrt{\frac{k}{n}}\big)~.$$
By Assumption \ref{ass:Cnk},
$$\sum_{k} P(W_{n}(\hat{f}_{n,k}) - W(\hat{f}_{n,k}) - C_n(k) > \epsilon + \sqrt{\frac{k}{n}}) \le \sum_{k} c_1 e^{-2c_o n(\epsilon + \sqrt{\frac{k}{n}})^2}  \le  e^{-2c_o n\epsilon^2}\sum_{k} c_1 e^{-2kc_o}~.$$
By setting
\begin{equation}
\label{eqnDel}
\Delta := \sum_{k} c_1 e^{-2kc_o} < \infty~,
\end{equation}
the result follows.
\end{proof}

\begin{proof}[{\bf Proof of Theorem \ref{thm:regretbound}}]
We follow the general strategy from \cite{bartlett2002}. For every $k$, we have
\begin{equation} \label{eq:regretbound1}
 W^*_{\mathcal{F}} - W(\hat{f}_n) = \left(W^*_{\mathcal{F}} - W^*_{\mathcal{F}_{k}}\right) + \left( W^*_{\mathcal{F}_{k}} - W(\hat{f}_n)\right)~.
\end{equation}
We first consider the second term in (\ref{eq:regretbound1}), and expand it as follows
\begin{equation} \label{eq:regretbound2}
W^*_{\mathcal{F}_{k}} - W(\hat{f}_n) = \left(W^*_{\mathcal{F}_{k}} - R_n(\hat{f}_n)\right) + \left(R_n(\hat{f}_n) - W(\hat{f}_n)\right)~.
\end{equation}
 By the definition of $R_n$, the first term of expression (\ref{eq:regretbound2}) is bounded by
$$W^*_{\mathcal{F}_{k}} - R_n(\hat{f}_n) \le W^*_{\mathcal{F}_{k}} - W_n(\hat{f}_{n,k}) + C_n(k) + \sqrt{\frac{k}{n}}~.$$
Fix $\delta > 0$, and choose some $f^*_k \in {\cal F}_k$ such that $W(f^*_k) + \delta \ge W^*_{\mathcal{F}_{k}}$.\footnotemark[6]
\footnotetext[6]{If the welfare criterion achieves its maximum on ${\cal F}_k$, then $f^*_k$ can be set equal to any maximizer. In general however such an optimum may not exist, and thus we must choose $f^*_k$ will to be an "almost maximizer" of the welfare criterion on ${\cal F}_k$.}.
We have
$$W^*_{\mathcal{F}_{k}} - W_n(\hat{f}_{n,k}) + C_n(k)+ \sqrt{\frac{k}{n}}\le W(f^*_k) + \delta - W_n(f^*_k) + C_n(k) + \sqrt{\frac{k}{n}}~.$$
Taking expectations of both sides and letting $\delta$ converge to $0$ yields
$$E[W^*_{\mathcal{F}_{k}} - R_n(\hat{f}_n)] \le E[C_n(k)] + \sqrt{\frac{k}{n}}~.$$
 By Lemma \ref{lemma:tailbound} and a standard integration argument \citep[see for instance problem 12.1 in][]{devroye1996}, the second term on the right hand side of (\ref{eq:regretbound2}) is bounded by
$$E[R_n(\hat{f}_n) - W(\hat{f}_n)] \le \sqrt{\frac{\log(\Delta e)}{2c_o n}}~.$$
Combining these bounds yields
$$E[W^*_{\mathcal{F}} - W(\hat{f}_n)] \le E[C_n(k)] + W^*_{\mathcal{F}} - W^*_{\mathcal{F}_{k}} + \sqrt{\frac{\log(\Delta e)}{2c_o n}} + \sqrt{\frac{k}{n}}~,$$
for every $k$, and our result follows.
\end{proof}

\begin{proof}[{\bf Proof of Lemma \ref{lem:rademacher}}]
We first establish the inequality
\begin{equation}
\label{eqnRad}
P(W_n(\hat{f}_{n,k}) - W(\hat{f}_{n,k}) - C_n(k) > \epsilon) \le \exp\left(-2n\epsilon^2(\frac{\kappa}{3M})^2\right)~.
\end{equation}
By two standard symmetrization arguments, we get
\begin{equation}
E\big[\sup_{f \in {\cal F}_k} W_n(f)-W(f)\big] \leq 2 E\big[ \sup_{f \in {\cal F}_k} \frac{1}{n} \sum_{i=1}^n \sigma_i f(Z_i)\big]=E\big[C_n(k)\big]~,
\label{eqna2}
\end{equation}
where we recall that $C_n(k)= E\left[2\sup_{f \in {\cal F}_k} \frac{1}{n} \sum_{i=1}^n \sigma_i f(Z_i)| Z_1,Z_2,\cdots,Z_n\right]$ and $\{\sigma_i\}_{i=1}^n$ is an $i.i.d$ sequence of Rademacher random variables independent from the data $\{Z_i\}_{i=1}^n$.
Note that
$$P(W_n(\hat{f}_{n,k}) - W(\hat{f}_{n,k}) - C_n(k) > \epsilon)\leq P(\sup_{f \in  {\cal F}_k}\left((W_n(f) - W(f)\right) - C_n(k) > \epsilon)~,$$
and set $M_{n,k}:= \sup_{f \in  {\cal F}_k}\left(W_n(f) - W(f)\right) - C_n(k)$. Combining the preceding inequality with (\ref{eqna2}) yields
$$P(W_n(\hat{f}_{n,k}) - W(\hat{f}_{n,k}) - C_n(k) > \epsilon)\leq P\left(M_{n,k}-EM_{n,k} >\epsilon \right)~.$$
To control the deviations of $M_{n,k}$ from its mean, we use McDiarmid's inequality \citep[see][Theorem 9.2; note that $M_{n,k}$ satisfies the bounded difference property with increments bounded by $\frac{3M}{n\kappa}$]{devroye1996} which yields the inequality
$$ P\left(M_{n,k}-EM_{n,k} >\epsilon \right)\leq \exp\left(-2n\epsilon^2(\frac{\kappa}{3M})^2\right)~,$$
from which our result follows.

The second inequality (where $C$ is a universal constant)
$$E[C_n(k)] \le C \frac{M}{\kappa}\sqrt{\frac{V_k}{n}}~,$$
follows from a chaining argument and a control on the universal entropy of VC subgraph classes \citep[see for instance the proof of Lemma A.4 in][]{kt2015}, along with Lemma \ref{lem:VCsbgf}.
\end{proof}

\begin{proof}[{\bf Proof of Lemma \ref{lem:holdout}}]
Let us assume for notational simplicity that the quantity $m = n(1-\ell)$ is an integer.
We first establish the inequality
\begin{equation}
\label{eqnHold}
P(W_m(\hat{f}_{m,k}) - W(\hat{f}_{m,k}) - C_m(k) > \epsilon) \le \exp\left(-2n\ell\epsilon^2(\frac{\kappa}{M})^2\right)~.
\end{equation}
By the definition of $C_m(k)$, we have
$$P(W(\hat{f}_{m,k}) - W(\hat{f}_{m,k}) - C_m(k) > \epsilon)=P(W_r(\hat{f}_{m,k}) - W(\hat{f}_{m,k}) > \epsilon)~.$$
Now, working conditionally on $\{Z_i\}_{i=1}^m$, we get by Hoeffding's inequality that
$$P(W_r(\hat{f}_{m,k}) - W(\hat{f}_{m,k}) > \epsilon | \{Z_i\}_{i=1}^m) \le \exp\left(-2n\ell\epsilon^2(\frac{\kappa}{M})^2\right)~.$$
Since the right hand side of the preceding inequality is non random, the inequality holds unconditionally as well.

We now establish the inequality
$$E[C_m(k)] \le C \frac{M}{\kappa\sqrt{(1-\ell)}}\sqrt{\frac{V_k}{n}}~.$$
By the definition of $C_m(k)$, we have
$$E[C_m(k)] = E[W_m(\hat{f}_{m,k}) - W_r(\hat{f}_{m,k})] = E[W_m(\hat{f}_{m,k}) - W(\hat{f}_{m,k}) + W(\hat{f}_{m,k}) - W_r(\hat{f}_{m,k})]~.$$
Note that by the law of iterated expectations, we have
$$E[W(\hat{f}_{m,k}) - W_r(\hat{f}_{m,k})] = 0~,$$
and by Lemma A.4 in \cite{kt2015} combined with Lemma \ref{lem:VCsbgf} there exists some universal constant $C$ such that:
$$E[W_m(\hat{f}_{m,k}) - W(\hat{f}_{m,k})] \le C\frac{M}{\kappa}\sqrt{\frac{V_k}{m}}~.$$
Since $m = (1-\ell)n $, the result follows.
\end{proof}

\begin{proof}[{\bf Proof of Propositions \ref{prop:rademacherbound} and \ref{prop:holdoutbound}}]
From the inequality
$$\frac{e^{-x}}{(1-e^{-x})} \leq \frac{1}{x}~,$$
and from (\ref{eqnDel}) and (\ref{eqnRad}), we derive that
$$\Delta \leq 1/2 \left(\frac{3M}{\kappa}\right)^2~.$$
Similarly, we derive from (\ref{eqnDel}) and (\ref{eqnHold}) that
$$\Delta \leq 1/(2l) \left(\frac{M}{\kappa}\right)^2~.$$
The results then follow by substituting these into the inequalities of Theorem \ref{thm:regretbound}.
\end{proof}

\begin{proof}[{\bf Proof of Theorem \ref{thm:eregretbound}}]
Our strategy here is to proceed analogously to the proof of Theorem \ref{thm:regretbound} with some additional machinery. Let $\hat{f}^e_n$ and $R_n^e(\cdot)$ be defined analogously to the case when the propensity score is known. For every $k$, we have that:
\begin{equation} \label{eq:eregretbound1}
 W^*_{\mathcal{F}} - W(\hat{f}^e_n) = \left(W^*_{\mathcal{F}} - W^*_{\mathcal{F}_{k}}\right) + \left(W^*_{\mathcal{F}_{k}} - W(\hat{f}^e_n)~\right).
\end{equation}
Adding and subtracting $R^e_n(\hat{f}^e_n)$ to the last term yields
\begin{equation} \label{eq:eregretbound2}
W^*_{\mathcal{F}_{k}} - W(\hat{f}^e_n) = \left(W^*_{\mathcal{F}_{k}} - R^e_n(\hat{f}^e_n)\right) + \left(R^e_n(\hat{f}^e_n) - W(\hat{f}^e_n)~\right).
\end{equation}
Let $f^*_k := \arg\max_{f \in \mathcal{F}_k} W(f)~,$ (if the supremum is not achieved, apply the argument to a $\delta$-maximizer of the welfare, and let $\delta$ tend to zero). Now consider the first term on the right hand side of (\ref{eq:eregretbound2}). Expanding yet again gives
\begin{equation} \label{eq:eregretbound3}
W^*_{\mathcal{F}_{k}} - R^e_n(\hat{f}^e_n) = W^*_{\mathcal{F}_{k}} - W_n(f^*_k) + W_n(f^*_k) - R^e_n(\hat{f}^e_n)~.
\end{equation}
From the definition of $R^e_n$, we have
$$W_n(f^*_k) - R^e_n(\hat{f}^e_n) \le W_n(f^*_k) - W^e_n(f^*_k) + C^e_n(k) + \sqrt{\frac{k}{n}} \le \frac{1}{n}\sum_{i=1}^n|\hat{\tau}_i - \tau_i| + C^e_n(k) + \sqrt{\frac{k}{n}}~.$$
Hence, considering the above inequality and taking expectations in (\ref{eq:eregretbound3}) yields
$$E[W^*_{\mathcal{F}_{k}} - R^e_n(\hat{f}^e_n))] \le E\Big[\frac{1}{n}\sum_{i=1}^n|\hat{\tau}_i - \tau_i|\Big] + E[C^e_n(k)] + \sqrt{\frac{k}{n}}~,$$
and thus by Assumption \ref{ass:ereg}
\begin{equation} \label{eq:eregretbound4}
E[W^*_{\mathcal{F}_{k}} - R^e_n(\hat{f}^e_n))] \le O(\phi_n^{-1})+ E[C^e_n(k)] + \sqrt{\frac{k}{n}}~.
\end{equation}
We now consider the second term on the right hand side of (\ref{eq:eregretbound2}). Let $\hat{k}$ be the class $k$ such that
$$\hat{f}^e_n = \hat{f}^e_{n,\hat{k}}~.$$
Note that $\hat{k}$ is random. We have
$$R^e_n(\hat{f}^e_n) - W(\hat{f}^e_n) = W^e_n(\hat{f}^e_{n,\hat{k}}) - C^e_n(\hat{k}) - \sqrt{\frac{\hat{k}}{n}} - W(\hat{f}^e_{n,\hat{k}})~.$$
By adding and subtracting $W_n(\hat{f}^e_{n,\hat{k}})$ and the function $\tilde{C}_n(\hat{k})$, we get

$$W^e_n(\hat{f}^e_{n,\hat{k}}) - C^e_n(\hat{k}) - \sqrt{\frac{k}{n}} - W(\hat{f}^e_{n,\hat{k}}) = $$
\begin{equation}\label{eq:egregretbound5}
\left(W^e_n(\hat{f}^e_{n,\hat{k}}) - W_n(\hat{f}^e_{n,\hat{k}})\right) + \left(\tilde{C}_n(\hat{k}) - C^e_n(\hat{k})\right)
+ \left(W_n(\hat{f}^e_{n,\hat{k}}) - W(\hat{f}^e_{n,\hat{k}}) -  \tilde{C}_n(\hat{k}) - \sqrt{\frac{\hat{k}}{n}}~\right)~.
\end{equation}
Note again that
$$\sup_k \big(W_n^e(\hat{f}^e_{n,k}) - W_n(\hat{f}^e_{n,k})\big) \le \frac{1}{n}\sum_{i=1}^n|\hat{\tau}_i - \tau_i|~,$$
and so by Assumptions \ref{ass:ereg} and \ref{ass:Cnke}, the first two terms of (\ref{eq:egregretbound5}) are of order $O(\phi^{-1}_n)$ in expectation. By the first part of Assumption \ref{ass:Cnke}, and an argument similar to the one used in the proof of Lemma \ref{lemma:tailbound}, it can be shown that
$$E\Big[\sup_k \big(W_n(\hat{f}^e_{n,k}) - W(\hat{f}^e_{n,k}) - \tilde{C}_n(k) - \sqrt{\frac{k}{n}}\big)\Big] \le \sqrt{\frac{\log(\Delta e)}{2c_0n}}~,$$
where $\Delta$ and $c_o$ are the same constants that appear in \ref{lemma:tailbound}. We thus get
\begin{equation} \label{eq:eregretbound6}
E[R^e_n(\hat{f}^e_n) - W(\hat{f}^e_n)] \le O(\phi^{-1}_n) +  \sqrt{\frac{\log(\Delta e)}{2m}}~.
\end{equation}
Now combining (\ref{eq:eregretbound4}) and (\ref{eq:eregretbound6}), we conclude that
$$E[W^*_{\mathcal{F}_{k}} - W(\hat{f}^e_n)] \le O(\phi^{-1}_n)+ E[C^e_n(k)] + \sqrt{\frac{k}{n}} +  \sqrt{\frac{\log(\Delta e)}{2m}}~.$$
Finally, by Assumption \ref{ass:Cnke}, we get
$$E[W^*_{\mathcal{F}} - W(\hat{f}^e_n)] \le O(\phi^{-1}_n) + E[\tilde{C}_n(k)] + W^*_{\mathcal{F}} - W^*_{\mathcal{F}_{k}} + \sqrt{\frac{k}{n}} +  \sqrt{\frac{\log(\Delta e)}{2m}}~,$$
for all $k$, and hence the result follows.
\end{proof}

\begin{proof}[{\bf Proof of Lemma \ref{lem:eholdout}}]
In what follows, we verify that the third condition of Assumption $\ref{ass:Cnke}$ is satisfied for the holdout penalty with estimated propensity score, as the first two conditions follow from previous arguments. 
Set
$$\tilde{C}_m(k)=W_m(\hat{f}^e_{m,k})-W_r(\hat{f}^e_{m,k})~.$$
Note that since the propensity score is unknown, the empirical welfare criteria $W_m$ and $W_r$ are infeasible. It can easily be shown that for this choice of $\tilde{C}_m(k)$, we have
$$|\tilde{C}_m(k)-C^e_m(k)|\leq \frac{1}{m}\sum_{i=1}^m |\hat{\tau_i}^E-\tau_i| + \frac{1}{r}\sum_{i=m+1}^n|\hat{\tau_i}^T-\tau_i|~,$$
which yields
$$E \sup_{k\geq 1}\left|\tilde{C}_m(k)-C^e_m(k)\right|=O(\phi^{-1}_n)~.$$
%For the Rademacher penalty,
%$$\tilde{C}_n(k)=E_{\sigma}\left[2 \sup_{f \in {\cal F}_k} \frac{1}{n} \sum_{i=1}^n \sigma_i f(Z_i)| Z_1,Z_2,\cdots,Z_n\right]~,$$
%which is the infeasible Rademacher penalty that depends on the unknown propensity score, then it can be shown that
%$$|\tilde{C}_n(k)-C^e_n(k)|\leq E_{\sigma}\left[\frac{2}{n}\sum_{i=1}^n |\hat{\tau}_i-\tau_i| \Big| Z_1, Z_2,\cdots,Z_n\right]~.$$
%Since the right hand side does not depend on $k$, we conclude that
%$$E \sup_{k\geq 1}\left|\tilde{C}_n(k)-C^e_n(k)\right|\leq 2 E\sum_{i=1}^n|\hat{\tau_i}-\tau_i|=O(\phi^{-1}_n)~,$$
%by Assumption \ref{ass:ereg}.
\end{proof}
 \begin{proof}[{\bf Proof of Proposition \ref{prop:monbias}}].
 Let $\mathcal{G}$ be the set of monotone allocations. Let $\pi_k$ denote the partition of $[0,1]$ formed by the points $x_i=i/2^k$, $i=0,\cdots, 2^k$. Let $\{\mathcal{G}_k\}_k$ be the approximating sequence defined in Example \ref{ex:aseq3}, and define $G^* \in \mathcal{G}$ to be a set such that $W(G^*) = W^*_\mathcal{G}$ (if no such $G^*$ exists, the argument proceeds by considering an ``almost maximizer"). By definition, for each $G \in \mathcal{G}$, there is an associated function $b_G:[0,1]\rightarrow [0,1]$ which determines the boundary of the allocation region, that is, such that $G = \{(x_1,x_2) \in \mathcal{X}:x_2 \le b_G(x_1)\}$.
 
Fix some $P \in \mathcal{P}_r$, where $\mathcal{P}_r$ is as defined in Assumption \ref{ass:monreg}. By definition,
$$W^*_\mathcal{G} - W^*_{\mathcal{G}_k} \le W(G^*) - W(\tilde{G}_k)~,$$
where $\tilde{G}_k \in \mathcal{G}_k$ is the allocation such that $b_{\tilde{G}_k}(\cdot)$ is the linear interpolation of $b_{G^*}$ on the partition $\pi_k$. We can re-write this as
\begin{equation}
\begin{split}
W(G^*) - W(\tilde{G}_k) &=  E\left[\left(\frac{YD}{e(X)}-\frac{Y(1-D)}{1-e(X)}\right)\cdot\left({\bf 1}\{X\in G^*\}-{\bf 1}\{X\in \tilde{G}_k\}\right)\right]\\
&\leq \frac{M}{\kappa}P_X(G^* \Delta \tilde{G}_k)~,
\end{split}
\end{equation}
where $\Delta$ denotes the symmetric difference operator, $A\Delta B:=A\backslash B \cup B\backslash A$.  Let $$M_i = [x_{i-1},x_i] \times [b_{G^*}(x_{i-1}),b_{G^*}(x_i)]~,$$
for $i = 1, \ldots, 2^k$. It follows from the monotonicity of $b_{G^*}$ that the graphs of the restrictions of $b_{G^*}(\cdot)$ and $b_{\tilde{G}_k}(\cdot)$ to $[x_{i-1},x_i]$ are contained in $M_i$. Hence we have that
$$P_X(G^* \Delta \tilde{G}_k) \le \sum_{i=1}^{2^k} P_X(M_i) = \sum_{i=1}^{2^k}P_X(M_{1i} \times M_{2i})~,$$
where $M_{1i} = [x_{i-1},x_i]$, $M_{2i} = [b_{G^*}(x_{i-1}),b_{G^*}(x_i)]$.
By Assumption \ref{ass:monreg},
$$P_X(M_{1i}\times M_{2i}) = \int_{M_{2i}}P_{X_1|x_2}(M_{1i})dP_{X_2} \le \frac{1}{2^k}AP_{X_2}(M_{2_i})~.$$
Summing over $i$:
$$\sum_{i=1}^{2^k}P_X(M_i) \le \sum_{i=1}^{2^k}\frac{1}{2^k}AP_{X_2}(M_{2i}) \le \frac{A}{2^k}~,$$
since the $\{M_{2i}\}_i$ form a partition of $[0,1]$. 
We thus obtain that
$$W^*_\mathcal{G} - W^*_{\mathcal{G}_k} \le A\frac{M}{\kappa}2^{-k}~,$$
as desired.
 \end{proof}

\section{Additional Results}\label{sec:supp_results}
\subsection{Supplement to Remark \ref{rem:tt1}}
In this subsection we provide some simple calculations that justify the comments made in Remark \ref{rem:tt1}. Consider first the Rademacher penalty, then Proposition \ref{prop:holdoutbound} shows that 
$$E_{P^n}[W^*_\mathcal{G} - W(\hat{G}_n)] \le \inf_k \Big[C\frac{M}{\kappa}\sqrt{\frac{V_k}{n}} + \big(W^*_{\mathcal{G}} - W^*_{\mathcal{G}_{k}}\big) + \sqrt{\frac{k}{n}}\Big]+ g(M,\kappa)\frac{M}{\kappa}\sqrt{\frac{1}{n}}~,$$
where $C$ is the universal constant derived in the bound of EWM in \cite{kt2015} and $g$ is defined as
$$g(M,\kappa) := 6\sqrt{\log\Big(\frac{3\sqrt{e}}{\sqrt{2}}\frac{M}{\kappa}\Big)}~.$$
Our first task is to quantify the size of $C$. By the proof of Lemma A.4. in \cite{kt2015}, we can see that the constant $C$ depends on a universal constant $K$ derived in Theorem 2.6.7 of \cite{van1996}, which establishes a bound on the covering numbers of a VC subgraph class. Inspection of the proof in \cite{van1996} allows us to conclude that a suitable $K$ is given by $K = 3\sqrt{e}/8$. Plugging this in to the expression for $C$ derived in \cite{kt2015} allows us to conclude that a suitable $C$ is given by $C = 36.17$. Turning to $g(M,\kappa)$, we can calculate that in order for it to surpass $C$ by an order of magnitude, we would need $M/\kappa$ to be about as large as $10^{120}$. This gives us a sense of the relative sizes of the terms in our bound.

\subsection{Supplement to Remark \ref{rem:holdout_test}}
In this subsection we perform a sample splitting exercise to estimate the welfare performance of various decision rules on the JTPA data. To estimate welfare, we split the data into two halves. The first half of the data (the ``estimating sample") is used to compute various decision rules. The second half of the data (the ``auxiliary sample") is used to estimate the welfare generated by each resulting treatment allocation.

Given a sample of size $n$ and a treatment allocation $G$, we estimate welfare using
\[\widehat{W}(G) = E_n\left[\frac{Y_iD_i}{e(X_i)}{\bf 1}\{X_i \in G\} + \frac{Y_i(1 - D_i)}{1 - e(X_i)}{\bf 1}\{X_i \notin G\}\right]~,\]
where $E_n(\cdot)$ is the sample average. We study the welfare performance of three decision rules: EWM on the class $\mathcal{G}_5$ as described in Section \ref{sec:application}, PWM with the holdout penalty on the sieve $\{\mathcal{G}_k\}_{k=1}^5$ as described in Section \ref{sec:application}, and a ``random baseline" which randomly assigns the same fraction of the population as PWM to job training. In Table \ref{tab:est_welfare} we report the estimated welfare computed on both the estimating and auxiliary samples. 
% Table generated by Excel2LaTeX from sheet 'Application Hypothesis Testing'
\begin{table}[!ht]
  \caption{Estimated Welfare Comparisons for JTPA Data}   \label{tab:est_welfare}%
  \begin{center}
\begin{threeparttable}
    \begin{tabular}{rccc}
\midrule
\midrule
   %  &       & \multicolumn{1}{l}{Treatment Rule} &  \\
          & PWM   & EWM   & Random Baseline \\ 
          & & & (average of 1000 draws) \\
    \midrule
        \multicolumn{1}{l}{Estimating Sample} & \$16,221 & \$16,522 & \$15,878 \\
    \multicolumn{1}{l}{Auxiliary Sample} & \$16,402 & \$16,272 & \$16,394 \\
          & (384)  & (395)  & (265) \\
\bottomrule
    \end{tabular}%
\begin{tablenotes}
\footnotesize
\item Standard errors in parentheses (see Remark \ref{rem:table_SE}).
\end{tablenotes}
\end{threeparttable}
\end{center}
\end{table}%

In Table \ref{tab:est_welfare}, we see that EWM has the highest estimated welfare when evaluated on the estimating sample. This is not surprising given that EWM maximizes empirical welfare on the estimating sample by construction. In contrast, when we estimate welfare using the auxiliary sample, we see that PWM has the highest estimated welfare, which shows that PWM can effectively protect against overfitting in this example. However, we stress that this difference was not found to be statistically significant (one-tailed p-value 0.34: see Remark \ref{rem:table_SE} below for details on how our test was constructed). We also note that the performance of PWM on the auxiliary sample is essentially the same as the performance of the random baseline rule; this is also the case when comparing EWM to a similar random baseline (not formally reported). It is possible that this is a feature specific to the monotone policy class (which we view as an exogenous constraint) in this application, and that a more flexible policy class would be able to outperform the random baseline via more selective targeting.

\begin{remark}\label{rem:table_SE}
In Table \ref{tab:est_welfare} we provide standard errors for the estimated welfare computed on the auxiliary sample. These should be interpreted as standard errors for the welfare estimate conditional on the estimated treatment allocation. To compute the standard errors we proceed as follows: given an auxiliary sample of size $m$ and a fixed treatment allocation $G$, it follows immediately by the Central Limit Theorem that
\[\sqrt{m}\left(\widehat{W}(G) - E[\widehat{W}(G)]\right) \xrightarrow{d} N(0, V(G))~,\]
as $m \rightarrow \infty$, where $V(G) = \text{Var}\left(\frac{YD}{e(X)}{\bf 1}\{X \in G\} + \frac{Y(1-D)}{1 - e(X)}{\bf 1}\{X \notin G\})\right)$. Let $\hat{V}(G)$ be the empirical analog of $V(G)$ computed on the auxiliary sample, then the standard error is given by $\sqrt{\hat{V}(G)/m}$.
By a similar argument, we can derive the limiting joint distribution for two distinct policies $G_1$ and $G_2$, which allows us to construct a difference-in-means test for the welfare difference between the two policies.
\end{remark}

\subsection{A Simulation Study}\label{sec:simulations}
In this section we perform a small simulation study to highlight the ability of the PWM rule to reduce $\mathcal{G}$-regret in an empirically relevant setting. We consider a situation where the planner has access to threshold-type allocations over five covariates, as described in Examples  \ref{ex:treatalloc2} and \ref{ex:aseq2}, and wishes to perform best-subset selection. The sieve sequence we consider is the same as in Example \ref{ex:aseq2}, where $\mathcal{G}_k$ is the set of threshold allocations on $k-1$ out of the $5$ covariates. For example, $\mathcal{G}_1$ contains only the allocations $G = \emptyset$ and $G = \mathcal{X}$, which correspond to threshold allocations that use zero covariates, $\mathcal{G}_2$ contains all threshold allocations on one out of the five covariates, etc. We focus here on the setting with five covariates for computational simplicity, but recent work by \cite{chen2016} suggests that solving this problem with ten or more covariates could be feasible in practice.

The problem that the planner faces is choosing how many covariates to use in the allocation: for example suppose that the distribution $P$ is such that some of the available covariates are irrelevant for assigning treatment. Of course, the planner could perform EWM on all the covariates at once, and by the bound in equation (\ref{eq:EWMregretbound}) this is guaranteed to produce small regret in large enough samples. However, if the sample is not large, the planner may be able to achieve a reduction in regret by performing PWM.  Through the lens of Corollary \ref{cor:knownclass}, our results say that PWM should behave \emph{as if} we had performed EWM in the smallest class $\mathcal{G}_k$ that contains all of the relevant covariates. 

We consider the following data generating process: Let $\mathcal{X} = [0,1]^5$, and $$X_i = (X_{1i},X_{2i},...,X_{5i}) \sim (U[0,1])^5~.$$ The potential outcomes for unit $i$ are specified as:
$$Y_i(1) = 50(2X_{2i} - (1-X_{1i})^4 - 0.5 + 0.5(X_{3i} - X_{4i})) + U_{1i}~,$$
$$Y_i(0) = 50(0.5(X_{3i} - X_{4i})) + U_{2i}~,$$
where $U_1$ and $U_2$ are distributed as $U[-80,80]$ random variables which are independent of each other and of $X$. The covariates enter the potential outcomes in three different ways:
\begin{itemize}
\item $X_{5i}$ is an irrelevant covariate; it does not play a role in determining potential outcomes at all. 
\item $X_{3i}$ and $X_{4i}$ affect both treatment and control equally; there will be a nonzero correlation between the observed outcome $Y_i$ and these covariates, but they serve no purpose for treatment assignment. 
\item $X_{1i}$ and $X_{2i}$ \emph{do} serve a purpose for assigning treatment, and both are used in the optimal threshold allocation. See Figure \ref{fig:sb} below.
\end{itemize}
Finally, the propensity score $P(D=1|X)$ is specified to be constant at $0.2$.

\begin{figure}[h]
\includegraphics[scale=0.6]{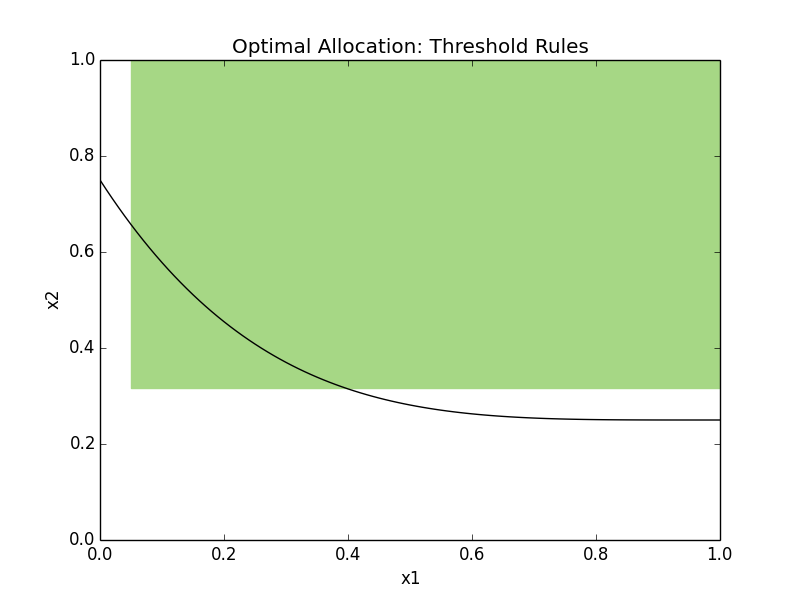} 
\centering
\captionsetup{justification = centering} 
\caption{Shaded in green: the best threshold-allocation for our design. Second-best welfare: 21.3 \\
Traced in black: the boundary of the first-best allocation.\label{fig:sb}}
\end{figure}

To implement PWM we used the holdout penalty, with $3/4$ of our sample designated as the estimating sample. In Appendix \ref{sec:appendixC} we explain in detail how to implement PWM as a mixed integer linear program. 

%For a decision rule $\hat{G}$, define the \emph{relative} $\mathcal{G}$-regret to be 
%$$\frac{E_{P^n}\big[W^*_\mathcal{G} - W(\hat{G})\big]}{W^*_\mathcal{G}}\cdot 100~.$$
Our results compare the $\mathcal{G}$-regret of the PWM rule against the regret of performing EWM in  $\mathcal{G}_6$ (which corresponds to the class that uses all five covariates) or performing EWM in $\mathcal{G}_3$ computed using $1000$ Monte Carlo iterations. Recall that $\mathcal{G}_3$ is the smallest class that contains the optimal threshold allocation. In light of Corollary \ref{cor:knownclass}, we would hope that PWM behaves similarly to doing EWM in $\mathcal{G}_3$ directly. In Figure \ref{fig:rr}, we plot the regret of these rules for various sample sizes.

\begin{figure}[h]
\includegraphics[scale=0.7]{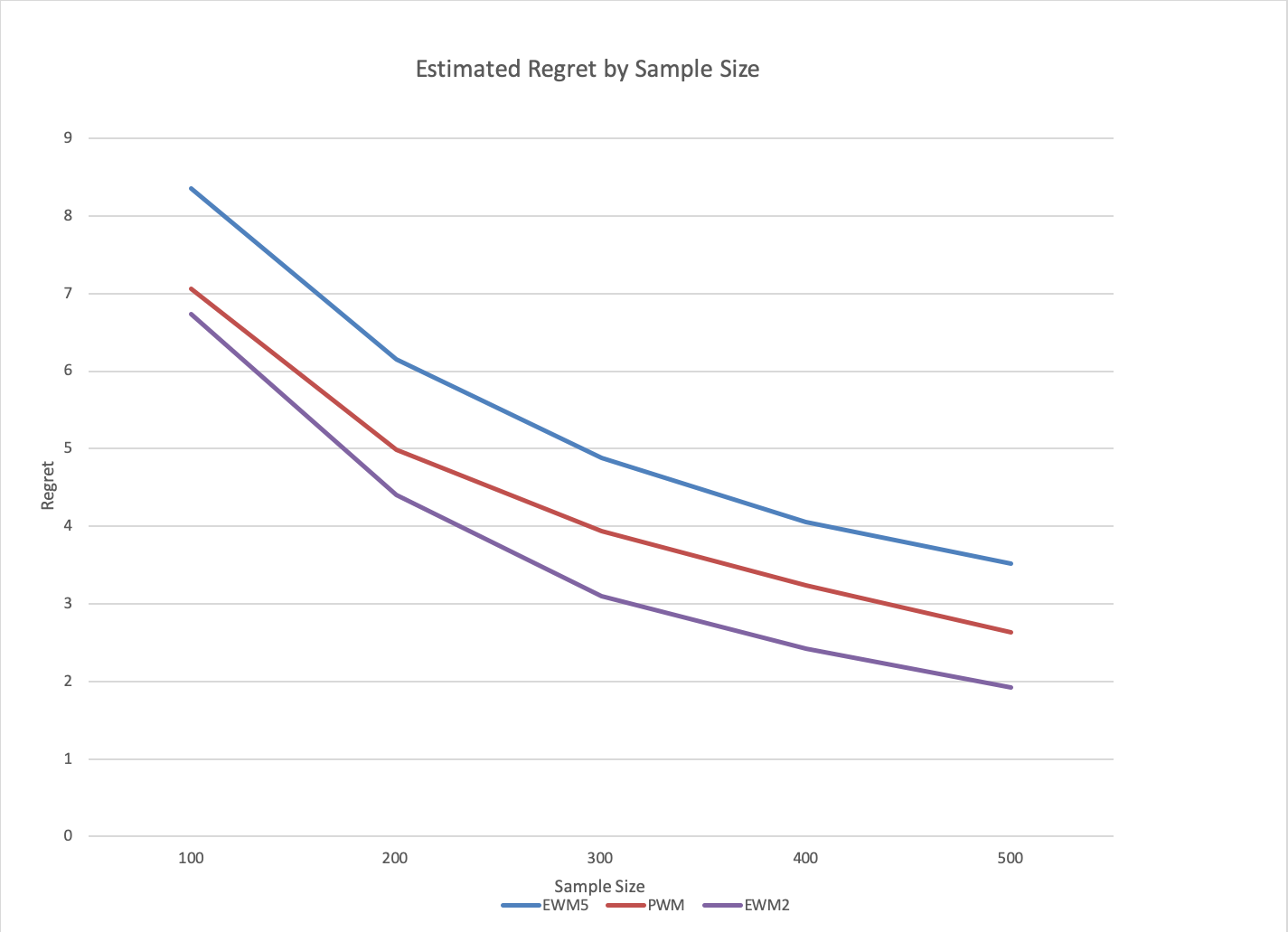} 
\centering
\captionsetup{justification = centering} 
\caption{Estimated regret by sample size. Optimal (second-best) welfare: 21.3. EWM5 corresponds to $\mathcal{G}_6$ (five covariates), EWM2 corresponds to $\mathcal{G}_3$ (two covariates).\label{fig:rr}}
\end{figure}

First we comment on the regret of performing EWM in $\mathcal{G}_6$ (recall that this corresponds to the set of allocations using all five covariates) vs. performing EWM in $\mathcal{G}_3$ (which corresponds to the set of allocations that use two of the five covariates).  As we would expect, regret decreases as sample size increases. Moreover, performing EWM in $\mathcal{G}_6$ results in larger regret at every sample size: performing EWM in $\mathcal{G}_3$ results in a 34\% improvement in regret relative to EWM in $\mathcal{G}_6$ on average, across the sample sizes we consider. 

Next, we comment on the performance of PWM. As we had hoped, the regret of PWM is smaller than the regret of performing EWM in $\mathcal{G}_6$ at every sample size: performing PWM results in a 19.8\% improvement in regret relative to EWM in $\mathcal{G}_6$ on average, across the sample sizes we consider. 
\subsection{Welfare Maximization with Entropy Restrictions on $\mathcal{G}$}\label{sec:entropy_bounds}
In this section we study the treatment choice problem when certain entropy restrictions are imposed on $\mathcal{G}$. First we derive an upper bound on the maximum regret of EWM under assumptions on the bracketing entropy of $\mathcal{G}$: 

Throughout this section let $\mathcal{X} = [0,1]^{d_x}$. Given a class of sets $\mathcal{G}$ of $\mathcal{X}$, let $\mathcal{H} := \{{\bf 1}_G: G \in \mathcal{G}\}$. Let $||\cdot||_p$ be the $L_p(\mu)$ metric on $\mathcal{H}$, where $\mu$ is Lebesgue measure on $\mathcal{X}$. Given $h_1, h_2 \in \mathcal{H}$, with $h_1 \le h_2$, let $[h_1, h_2] := \{h \in \mathcal{H}: h_1 \le h \le h_2\}$. We call the set $[h_1,h_2]$ a \emph{bracket}. Given $\epsilon > 0$, define $N_p^B(\epsilon, \mathcal{H},\mu)$ to be the smallest $k$ such that for some pairs $(h_j^L, h_j^U)$, $j = 1, ... ,k$ $\in \mathcal{H}$, with $h_j^L \le h_j^U$ and $||h_j^U - h_j^L||_p < \epsilon$,
\[\mathcal{H} \subset \bigcup_{j=1}^k[h_j^L,h_j^U]~.\]
We call $H_p^B(\epsilon,\mathcal{H},\mu) := \log N_p^B(\epsilon,\mathcal{H},\mu)$ the $L_p(\mu)$ bracketing entropy \citep[in the sense of][]{alexander1984}.

Given this definition, we impose the following assumption on the bracketing entropy of $\mathcal{G}$:
\begin{assumption}\label{ass:bracket_entropy}
There exist positive constants $K$, $r$ for which
\[H_1^{B}(\epsilon, \mathcal{H}, \mu) \le K\epsilon^{-r}~,\]
for all $\epsilon > 0$.
\end{assumption}

\cite{dudley1999} provides many examples for which this assumption holds. In particular, by Theorem 8.3.2 in \cite{dudley1999}, if $\mathcal{G}$ is the set of monotone allocations in $[0,1]^{d_x}$, then Assumption \ref{ass:bracket_entropy} holds with $r = d_x - 1$ \citep[and the brackets can be constructed in the sense of][]{alexander1984}.

As we have emphasized throughout the paper, to obtain bounds on maximum regret for classes of infinite VC dimension, we must impose additional regularity conditions on the DGP. To that end, we consider the following assumption:

\begin{assumption}\label{ass:bounded_density}
Let $\mathcal{P}_r$ be a set of DGPs such that there exists some constant $A > 0$, where for every distribution in $\mathcal{P}_r$, the distribution of $X$ is continuous with density bounded above by $A$.
\end{assumption}

With this additional regularity condition, we obtain the following upper bound on maximum regret for EWM:

\begin{proposition}\label{prop:upper_bound} Under Assumptions \ref{ass:so}, \ref{ass:BO}, \ref{ass:bracket_entropy}, and \ref{ass:bounded_density}, we have that
$$\sup_{P \in \mathcal{P}_r \cap \mathcal{P}(M,\kappa)} E_{P^n}[W(G^*) - W(\hat{G}_{EWM})] = O(\tau(n))~,$$
where $\tau(n) = n^{-1/2}$ if $r < 1$, $\tau(n) = \log(n)/\sqrt{n}$ if $r = 1$, and $\tau(n) = n^{-1/(1+r)}$ if $r > 1$.
\end{proposition}

Note that this result \emph{does not} assume that the first-best allocation is contained in $\mathcal{G}$. From Proposition \ref{prop:upper_bound} we see that for $r$ sufficiently small, EWM converges at a parametric rate (under suitable regularity conditions). Similar results have been obtained in the classification context by \cite{mammen1999} and \cite{tsybakov2004}.

Next, we present a lower-bound on maximum regret under the following assumption on the $L_1(\mu)$ $\epsilon$-capacity:

Given $\epsilon > 0$, define $D_p(\epsilon, \mathcal{H}, \mu)$ to be the largest $k$ such that there exist functions $h_1, ... h_k \in \mathcal{H}$ with $||h_i - h_j||_p > \epsilon$ for $i \ne j$. We call $H_p(\epsilon,\mathcal{H},\mu) := \log D_p(\epsilon,\mathcal{H}, \mu)$ the $L_p(\mu)$ \emph{epsilon-capacity}.

Given this definition, we impose the following assumption on the $\epsilon$-capacity of $\mathcal{G}$:

\begin{assumption}\label{ass:entropy}
There exist positive constants $K_1,$ $K_2$, $\epsilon_1 > 0$, $r \ge 1$ such that 
\[K_2\epsilon^{-r} \le H_1(\epsilon, \mathcal{H}, \mu) \le K_1\epsilon^{-r}~,\]
for all $0 < \epsilon \le \epsilon_1$.
\end{assumption}

It can be shown that if $\mathcal{G}$ satisfies Assumption \ref{ass:bracket_entropy}, then the upper bound in Assumption \ref{ass:entropy} will also hold. However, the reverse may not be true. \cite{dudley1999} provides many examples for which Assumption \ref{ass:entropy} holds, and in particular it holds for the set of monotone allocations in $[0,1]^{d_x}$ with $r = d_x - 1$ (see Theorem 8.3.2). 

With this assumption we obtain the following lower bound on maximum regret:

\begin{proposition}\label{prop:lower_bound}
Let $\mathcal{P}^*(\mu) \subset \mathcal{P}(M,\kappa)$ be the set of DGPs such that the marginal distribution of $X$ is $\mu$, and $G^* = G^*_{FB}$.
Under Assumption \ref{ass:entropy}, there exists a positive constant $B$ (which depends on $M$, $K_1, K_2$, $r$), such that
\[ \inf_{\hat{G}}\sup_{P \in \mathcal{P}^*(\mu)} E_{P^n}[W(G^*) - W(\hat{G})] \ge Bn^{-1/(1+r)}~, \]
for all $n \ge 1$.
\end{proposition}

For classes $\mathcal{G}$ such that Assumptions \ref{ass:bracket_entropy} and \ref{ass:entropy} \emph{both} hold, Propositions \ref{prop:upper_bound} and \ref{prop:lower_bound} immediately imply the following rate optimality result for EWM:

\begin{corollary}\label{cor:rate_optimal}
Given Assumptions \ref{ass:bracket_entropy}, \ref{ass:bounded_density}, and \ref{ass:entropy}, EWM is rate-optimal over $\mathcal{P}_r \cap \mathcal{P}(M,\kappa)$ for $r > 1$ and rate-optimal up to a $\log$ factor for $r = 1$. 
\end{corollary}
\begin{proof}
This follows immediately from the fact that $\mathcal{P}^*(\mu) \subset \mathcal{P}_r \cap \mathcal{P}(M,\kappa)$, and hence
\[\inf_{\hat{G}}\sup_{P \in \mathcal{P}_r \cap \mathcal{P}(M,\kappa)} E_{P^n}[W(G^*) - W(\hat{G})] \ge \inf_{\hat{G}}\sup_{P \in \mathcal{P}^*(\mu)} E_{P^n}[W(G^*) - W(\hat{G})]~.\]
\end{proof}

As we remarked above, for the set of monotone allocations on $[0,1]^2$, Assumptions \ref{ass:bracket_entropy} and \ref{ass:entropy} hold with $r = 1$. Hence we can conclude that EWM is rate-optimal up to a $\log$-factor for monotone allocations when the distribution of $\mathcal{X}$ is continuous with a bounded density. Note that Corollary \ref{cor:rate_optimal} only establishes rate optimality when $r$ is sufficiently large. For $r < 1$, the lower bound presented in Proposition \ref{prop:lower_bound} is certainly too loose: the set of DGPs used in the proof of Proposition \ref{prop:lower_bound} impose a ``hard margin", and hence converge much faster than the parametric rate when $r < 1$. 

\begin{remark}\label{rem:triv_optimal}
It can be shown that the PWM procedure implemented as in Section \ref{sec:application} can also achieve the rate established in Corollary \ref{cor:rate_optimal} (up to a log factor). To see why, note that by using arguments similar to those used in the proof of Propositions \ref{prop:upper_bound}, it can be shown that for the holdout penalty,
\[\sup_{P \in \mathcal{P}_r\cap\mathcal{P}(M,\kappa)}E[C_m(k)] = O\left(\frac{\log(n)}{\sqrt{n}}\right)~.\]
Combining this result with Proposition \ref{prop:monbias} and Corollary \ref{cor:regretrate} we get that the maximum regret of PWM is bounded above by (up to constants),
\[\frac{\log(n)}{\sqrt{n}} + \inf_k\left(2^{-k} + \sqrt{\frac{k}{n}}\right)~,\]
whose rate of convergence is dominated by the leading term.
\end{remark}

\section{Computational Details}\label{sec:appendixC}
In this section we provide details on how we perform the computations of Section \ref{sec:application} and Appendix \ref{sec:simulations}. All of our work is implemented in Python paired with Gurobi. We begin with Section \ref{sec:application}, then proceed to Appendix \ref{sec:simulations}.

\subsection{Application Details}
First we describe how to compute each $\hat{G}_{n,k}$ to solve PWM over monotone allocations. Recall the definition of $\psi_{T,j}(x)$ as defined in Example \ref{ex:aseq3}. We modify this definition to accommodate the fact that our covariates do not lie in the unit interval. In particular, we restrict ourselves to levels of education that lie in the interval $[5,20]$, which leads to the following modification.

\[ \psi_{T,j}(x) = \begin{cases} 
      1 - |\frac{T}{15}(x-5) - j|, & x \in \big[\frac{j-1}{T/15}+5,\frac{j+1}{T/15}+5\big] \cap [5,20] \\
      0, & \text{otherwise}~.
   \end{cases}
\]

Let $\Theta_T = \begin{bmatrix}\theta_0 & \theta_1 &\cdots &\theta_T\end{bmatrix}'$ and let $\bar{\Theta}_T = \begin{bmatrix} -1 & \theta_0 &\theta_1& \cdots &\theta_T\end{bmatrix}'$. Let our two dimensional covariate be denoted as $x = (x^{(1)},x^{(2)})$ where $x^{(1)}$ is level of education and $x^{(2)}$ is previous earnings. Let $$\Psi_T(x) = \begin{bmatrix}x^{(2)} & \psi_{T,0}(x^{(1)}) &\cdots &\psi_{T,T}(x^{(1)})\end{bmatrix}'~.$$ To compute $\hat{G}_{n,k}$ we solve the following mixed integer linear program (MILP), which modifies the MILP described in \cite{kt2015} for ``Single Linear Index Rules":
\begin{align*}
\underset{\substack{\theta_0,\theta_1,...,\theta_T,\\
z_1,...,z_n}}{\text{max}} \qquad &
   \sum_{i=1}^n \tau_i \cdot z_i
\\[2ex]
\text{subject to}\qquad & \frac{\bar{\Theta}_T'\Psi_T(x_i)}{c_{iT}} < z_i \leq \frac{\bar{\Theta}_T'\Psi_T(x_i)}{c_{iT}} + 1, \; i = 1, \ldots, n \\
& z_i \in \{0,1\}, \; i= 1, \ldots, n\\
& D_T\Theta_T \le 0
\end{align*}
where $T = 2^{k-1}$, $\tau_i$ is as defined in equation (\ref{eq:EWMobjective}), $c_T$ is an appropriate constant (to be discussed in the following sentence), and $D_T$ is the differentiation matrix as defined in Example \ref{ex:aseq3}. $c_{iT}$ is a constant chosen such that $c_{iT} > \sup_{\Theta_T} |\bar{\Theta}_T'\Psi_T(x_i)|$, which allows us to formulate a set of what are known as ``big-M" constraints. To implement such a constraint it must necessarily be the case that $\Theta_T$ is bounded, so in order to implement PWM we also include an implicit (very large) bound on the possible treatment allocations.\footnotemark[7]

\footnotetext[7]{Big-M constraints have the potential to cause numerical instabilities when solving MILPs that are poorly formulated. We found that it was important to ensure that the covariates are scaled to within the same order of magnitude and that the {\tt IntFeasTol} and {\tt FeasibilityTol} parameters in Gurobi were set to their smallest possible values.}

The first two sets of constraints impose that the treatment allocation result in a piecewise linear boundary, the third set of constraints impose that this boundary is monotone. The strength of this formulation is that it imposes monotonicity via a \emph{linear} constraint, which allows us to solve the problem as a MILP.

 \subsection{Simulation Details}
We describe a MILP to compute each $\hat{G}_{n,k}$ over threshold allocations on $d$ covariates. Define $x$ to be a $(d+1)$-dimensional vector where $x = (1, x^{(1)}, x^{(2)}, ..., x^{(d)})$, with the last $d$ components denoting the $d$ covariates, and suppose $x \in [0,1]^{d+1}$, which is the case in the simulation design. We define the threshold $\beta_k$ on covariate $x^{(k)}$ to be a $(d+1)$-dimensional vector such that the first component is in $[-1,1]$, all other components other than the $(k+1)$st are zero, and the $(k+1)$st component is one of $\{-1, 0, 1\}$. Let $A = \{1,2,...,d\}$ index the dimension of the threshold. We modify the MILP described in \cite{kt2015} for ``Multiple Linear Index Rules":

\begin{align*}
\underset{\substack{\{\beta_a\}_{a\in A},\\
\{z^a_1,...,z^a_n\}_{a\in A},z^*_1,...,z^*_n}}{\text{max}} \qquad &
   \sum_{i=1}^n \tau_i \cdot z_i^*
\\[2ex]
\text{subject to}\qquad & \frac{x_i'\beta_a}{c} < z^a_i \leq \frac{x_i'\beta_a}{c} + 1, \; i = 1, \ldots, n, \; a\in A \\
& 1- |A| + \sum_{a \in A}z^a_i \le z_i^* \le \frac{1}{|A|}\sum_{a \in A}z^a_i, \; i = 1, \ldots, n\\
& \beta_a^{(1)} \in [-1,1], \; a \in A\\
& \beta_a^{(j)} = 0, \; j > 1, j \ne a+1, \; a \in A \\
& \sum_{a \in A} e_a = k \\
& -e_a \le \beta_a^{(1)} \le e_a, \; a \in A\\
& \beta_a^{(a+1)} = y_{a, 1} - y_{a, 2}, \; a \in A\\
& y_{a, 1} + y_{a, 2} = e_a, \; a \in A\\
& \{z^a_i\}_{a \in A}, z^*_i \in \{0,1\}, \; i= 1, \ldots, n\\
& \{e_a\}_{a\in A} \in \{0, 1\}, \; a \in A\\
&\{y_{a, 1}\}_{a\in A}, \{y_{a, 2}\}_{a\in A} \in \{0, 1\}, \; a \in A
\end{align*}

The constraints serve the following roles: the first two constraints enforce the assignment of observations to treatment, the next two constraints enforce part of the structure of the threshold allocation, the fifth constraint specifies that only $k$ thresholds can be used, and the three subsequent constraints enforce this. Again we require an appropriately chosen constant $c$ to implement a set of big-M constraints, but in this case the choice is straightforward: $c = d+2$ will suffice since this guarantees that $c > x_i'\beta_a$ for any possible $x_i$ and $\beta_a$, by construction.

\begin{remark}\label{rem:speed-up}
Solving the above program for the simulation design of Appendix \ref{sec:simulations} with a sample size of $2000$ took approximately one hour and fifteen minutes on a 2018 iMac.  In practice, the solution of this MILP could potentially be further optimized using the improvements developed in \cite{bertsimas2016} and \cite{chen2016}. Alternatively, careful implementation of a direct parameter search could also considered: see for example the work in \cite{zhou2018} using a tree-based policy class.
\end{remark}

\section{Proofs for Appendix B}\label{sec:entropy_proofs}
\noindent {\bf Proof for Proposition \ref{prop:upper_bound}}
\begin{proof}
We follow the general strategy of Theorem 1 in \cite{mammen1999}. Let $\bar{W}(\cdot) = (\kappa/M)W(\cdot)$ be a normalized version of $W(\cdot)$. Let $G^*$ be a maximizer of $\bar{W}(\cdot)$ in $\mathcal{G}$. Let $\mathcal{P} = \mathcal{P}_r \cap \mathcal{P}(M,\kappa)$ and define
$$T_n = \sqrt{n}\frac{\bar{W}(G^*)-\bar{W}(\hat{G}) - (\bar{W}_n(G^*) - \bar{W}_n(\hat{G}))}{q_n}~,$$
where $q_n = 1$ if $r < 1$, $q_n = \log(n)$ if $r = 1$ and $q_n = n^{(r-1)/2(r+1)}$ if $r > 1$. By the definition of $\hat{G}$, and $G^*$ we have that $\bar{W}_n(\hat{G}) \ge \bar{W}_n(G^*)$, $\bar{W}(\hat{G}) \le \bar{W}(G^*)$, and hence we have that
$$0 \le \sqrt{n}q_n^{-1}\left(\bar{W}(G^*) - \bar{W}(\hat{G})\right) \le T_n~.$$
Now we argue that $E[T_n]$ is uniformly bounded over $\mathcal{P}$ for $n$ sufficiently large, which given the definition of $q_n$ implies the statement of the theorem. To that end, note that
$$E[T_n] \le E[S_n]~,$$
where
$$S_n  = \sup_{G \in \mathcal{G}}\sqrt{n}q_n^{-1}|\bar{W}(G^*) - \bar{W}(G) - (\bar{W}_n(G^*) - \bar{W}_n(G))|$$
$$ = \sup_{G \in \mathcal{G}}\sqrt{n}q_n^{-1}\left|\frac{1}{n}\sum_{i=1}^n\left(\bar{g}(Z_i)({\bf 1}\{X_i \in G^*\} - {\bf 1}\{X_i \in G\}) - E[\bar{g}(Z_i)({\bf 1}\{X_i \in G^*\} - {\bf 1}\{X_i \in G\})]\right)\right|~,$$
with
$$\bar{g}(Z_i) = \frac{\kappa}{M}g(Y_i, D_i, X_i) = \frac{\kappa}{M}\left(\frac{Y_iD_i}{e(X_i)} - \frac{Y_i(1-D_i)}{1 - e(X_i)}\right)~.$$
For the case $r < 1$, we can invoke Lemma \ref{lem:small_tail_ineq} to conclude immediately that:
\[\sup_{P \in \mathcal{P}}E[S_n] = O(1)~,\]
so we are done. 
For the case $r \ge 1$, note that $S_n \le 2\sqrt{n}/q_n$, which gives that for any $D > 0$,
\[E[S_n] \le D + 2\frac{\sqrt{n}}{q_n}P(S_n > D)~,\]
hence we can apply Corollary \ref{cor:tail_ineq} to the last probability to conclude that $\sup_{P\in\mathcal{P}}E[S_n]  = O(1)$. Let $\mathcal{\widetilde{F}} = \{\bar{g}\cdot({\bf 1}_{G^*} - {\bf 1}_G) : G \in \mathcal{G}\}$, then for an appropriate choice of $D$ (which depends on $P$ \emph{only} through $K$ and $r$, and $A$), Corollary \ref{cor:tail_ineq} gives
\[P(S_n > D) \le C\exp(-D^2q_n^2)~,\]
for some constant $C$ which depends on $P$ \emph{only} through $K$, $r$, and $A$. Hence we can conclude that
\[\sup_{P\in\mathcal{P}}E[T_n] \le \sup_{P\in\mathcal{P}}E[S_n]  = O(1)~,\]
as desired.
\end{proof}

\noindent{\bf Proof of Proposition \ref{prop:lower_bound}}
\begin{proof}
Define
\[L_n : =  \inf_{\hat{G}}\sup_{P \in \mathcal{P}^*(\mu)} E_{P^n}[W(G^*) - W(\hat{G})]~.\]
We follow the general strategy of Theorem 6 in \cite{massart2006}. For every $h \in \mathcal{H} = \{{\bf 1}_G : G \in \mathcal{G}\}$, set $\tau_h(x) = (M/4)(2h(x) - 1)$, $\gamma_h(x) = (2/M)\tau_h(x)$, and define $P_h$ as the joint distribution on $\mathcal{X}\times\{0,1\}\times\mathcal{Y}^2$ (i.e. the set of realizations of $(X,D,Y(1),Y(0))$) such that under $P_h$, $X$ has distribution $\mu$, 
\[
  Y(1)|\{X=x\} =
  \begin{cases}
                                   \frac{M}{2} & \text{with prob. $\frac{1 + \gamma_f(x)}{2}$} \\
                                   -\frac{M}{2} & \text{with prob. $\frac{1 - \gamma_f(x)}{2}$} \\
  \end{cases}
\]
$Y(0)|\{X=x\} = 0$, and $D$ is Bernoulli$(0.5)$ independent of everything else. Note that by construction we have that $\tau_h(x) = E_{P_h}[Y(1) - Y(0)|X=x] = \tau(x)$, $h$ describes the first-best decision rule under $P_h$, and $P_h \in \mathcal{P}^*(\mu)$. 

Next, let $\mathcal{C}$ be a finite subset of  $\mathcal{H}$, then it follows that:
$$\inf_{\hat{G}}\sup_{P \in \mathcal{P}^*(\mu)} E_{P^n}[W(G^*) - W(\hat{G})] \ge \inf_{\hat{G}}\sup_{h \in \mathcal{C}} E_h[W(G^*) - W(\hat{G})]~,$$
where $E_h = E_{P^n_h}$. Since, under $P_h$, $G^*$ is the first best allocation by construction, we get that
\[W(G^*) - W(G) = \int_{G^* \Delta G}|\tau(X)|dP_X~,\]
for any $G \in \mathcal{G}$. Hence it follows that, given the construction of $\tau_h$:
\[W(G^*) -W(G) = \frac{M}{4} \mu(G^* \Delta G)~.\]
Putting all this together and using the fact that $h = {\bf 1}_{G^*}$ under $P_h$:
\[L_n \ge \inf_{\hat{h}\in\mathcal{H}}\sup_{h \in \mathcal{C}}\frac{M}{4}E_{h}[||h - \hat{h}||_1]~,\]
%\inf_{\hat{f}\in\mathcal{H}}\sup_{f \in \mathcal{C}}E_{f}[||\hat{f} - f||_1]~,$$
where
$\hat{h} = {\bf 1}_{\hat{G}}$, and $||\cdot||_1$ is the $L_1(\mu)$ norm. Define the statistic
$$\tilde{h} = \arg\min_{h \in \mathcal{C}} ||h - \hat{h}||_1~,$$
then by the triangle inequality it follows that 
$$\inf_{\hat{h}\in\mathcal{H}}\sup_{h \in \mathcal{C}}E_{h}[||\hat{h} - h||_1] \ge \frac{1}{2}\inf_{\tilde{h}\in\mathcal{C}}\sup_{h \in \mathcal{C}}E_{h}[||\tilde{h} - h||_1]~.$$
We now construct the appropriate set $\mathcal{C}$. Let $\mathcal{C}'$ be an $\epsilon$-packing set of $\mathcal{H}$, and let $\mathcal{C}''$ be a $C\epsilon$ cover of $\mathcal{H}$ for some $C >1$, $\epsilon > 0$ to be specified later. By definition, each $h \in \mathcal{C}'$ lies in some ball of radius $C\epsilon$ centered at a point in $\mathcal{C}''$. So by taking $\mathcal{C}$ to be the intersection of $\mathcal{C}'$ with such a ball in $\mathcal{C}''$  which results in a set of maximal cardinality, we get that for $h_1, h_2 \in \mathcal{C}$, where $h_1 \ne h_2$, 
$$\epsilon \le ||h_1 - h_2||_1 \le C\epsilon~,$$ 
and moreover,
$$\log(|\mathcal{C}|) \ge H_1(\epsilon,\mathcal{H},\mu) - H_1(C\epsilon,\mathcal{H},\mu)~.$$
To see this, note that since we have constructed $\mathcal{C}$ to have maximal cardinality, it must be the case that
$$|\mathcal{C}| \ge \frac{|\mathcal{C}'|}{|\mathcal{C}''|}~,$$
and by definition, $|\mathcal{C}'| = H_1(\epsilon,\mathcal{H},\mu)$, $|\mathcal{C}''| \le H_1(C\epsilon,\mathcal{H},\mu)$.

Now, by Markov's inequality,
$$\inf_{\tilde{h}\in\mathcal{C}}\sup_{h \in \mathcal{C}}E_{h}[||\tilde{h} - h||_1] \ge \epsilon\inf_{\tilde{h} \in \mathcal{C}}(1 - \inf_{h\in\mathcal{C}}P^n_h(\tilde{h} = h))~,$$
and hence by Lemma 8 in \cite{massart2006},

$$L_n \ge \frac{M\epsilon}{8}(1 - \alpha)~,$$

where $\alpha := 0.71$, as long as $\mathcal{\bar{K}} \le \alpha\log(|\mathcal{C}|)$, where, for some fixed $h_0 \in \mathcal{C}$:
$$\mathcal{\bar{K}} := \frac{1}{|\mathcal{C}| - 1}\sum_{h \in \mathcal{C}, h \ne h_0}\mathcal{K}(P_h^n,P_{h_0}^n)$$
$$ = \frac{n}{|\mathcal{C}| - 1}\sum_{h \in \mathcal{C}, h \ne h_0}\mathcal{K}(P_h,P_{h_0})~,$$
and $\mathcal{K}(\cdot,\cdot)$ is the Kullback-Leibler divergence. By Lemma \ref{lemma:KL_calculate} we have that
$$\mathcal{\bar{K}} \le n\sup_{h \in \mathcal{C}}||h - h_0||_1 \le n\epsilon~,$$
where the last inequality follows by the construction of $\mathcal{C}$. 

Again by the construction of $\mathcal{C}$, we can choose $C$ such that there exists some positive constant $C_1$ for which $\log(|\mathcal{C}|) \ge C_1\epsilon^{-r}$ for $\epsilon  \le \epsilon_1$, and therefore
$$\frac{\mathcal{\bar{K}}}{\log|\mathcal{C}|} \le \frac{n}{C_1}\epsilon^{1+r}~.$$
Hence we can conclude that $L_n \ge (M\epsilon/8)(1 - \alpha)$ whenever
$$\frac{n}{C_1}\epsilon^{1+r} \le \alpha~,$$
that is,
$$\epsilon \le \left(\alpha C_1\right)^{1/(1+r)}n^{-1/(1+r)}~.$$
Now, we may also choose $C$ such that $\alpha C_1\le \epsilon_1^{1+r}$, so that the constraint $\epsilon \le \epsilon_1$ is satisfied if we set
$$\epsilon = \left(\alpha C_1\right)^{1/(1+r)}n^{-1/(1+r)}~.$$
Hence we have that
$$L_n =  \inf_{\hat{G}}\sup_{P \in \mathcal{P}^*(\mu)} E_{P^n}[W(G^*) - W(\hat{G})]  \ge An^{-1/(1+r)}~,$$
where $A$ is a constant which depends on $K_1$, $K_2$, $M$, and $r$ as desired.
\end{proof}

\begin{proposition}\label{prop:alexander_ineq}
Let $\{Z_i\}_{i=1}^n$ be a sequence of i.i.d random vectors with distribution $P$. Let $Z = (Z_1,Z_2)$, and let $\mathcal{F}$ be a class of real-valued functions of the form $f(z) = f(z_1,z_2) = g(z)\cdot h(z_2)$, where $h \in \mathcal{H}$, $\mathcal{H}$ is a class of functions with values in $\{0,1\}$, and $g$ is some fixed real-valued function (which may depend on $P$) such that $|g| \le 1$. Let $P_2$ be the marginal distribution of $Z_2$ and suppose $\mathcal{H}$ satisfies
 \begin{equation}\label{eq:2brackets}
  H_2^B(\epsilon, \mathcal{H}, P_2) \le K\epsilon^{-\ell}~,
  \end{equation}
for some constants $K > 0$, $\ell \ge 2$, for all $\epsilon > 0$. Then there exist positive constants $C_1, C_2, C_3, C_4$ (which depend only on $K$ and $\ell$) such that if
 \begin{equation}\label{eq:xi_upper}
  \xi \le \frac{\sqrt{n}}{128}~,
  \end{equation}
and 
 \begin{equation}\label{eq:xi_lower}
\xi \ge \begin{cases} 
          C_1n^{(\ell - 2)/2(\ell+2)},& \ell \ge 2 \\
          C_2\log \max(n,e), & \ell = 2 \\
           \end{cases}
  \end{equation}
then
\[P^n\left(\sup_{f \in \mathcal{F}} \left|\frac{1}{\sqrt{n}}\sum_{i=1}^n[f(Z_i) - Ef(Z_i)]\right| > \xi\right) \le C_4\exp(-\xi^2)~.\]
\end{proposition}

\begin{proof}
We follow the general strategy of Theorem 2.3 and Corollary 2.4 in \cite{alexander1984}. Let
\[\nu_n(f) = \frac{1}{\sqrt{n}}\sum_{i=1}^n[f(Z_i) - Ef(Z_i)]~.\]
We begin with a series of definitions. Let $\delta_0 > \delta_1 > ... > \delta_N > 0$ be a sequence of real numbers where $\{\delta_j\}_j$ and $N$ are to be specified precisely later in the proof. For every $0 \le j \le N$, there exists a set of $\delta_j$-brackets $\mathcal{H}^B_j$ of $\mathcal{H}$ such that $|\mathcal{H}^B_j| = N_2^B(\delta_j, \mathcal{H}, P_2)$. For each $h \in \mathcal{H}$ let $(h_j^L, h_j^U) \in \mathcal{H}^B_j$ be the brackets such that $h_j^L \le h \le h_j^U$ and $||h_j^H - h_j^L||_2 < \delta_j$. Define the function $H_{\theta}(\cdot): (0,\infty) \rightarrow [0, \infty)$ as follows:

    \[ H_{\theta}(u) = \begin{cases} 
          Ku^{-\ell}, &u \le 1 \\
          -\frac{K}{\theta}u + \frac{K(1+\theta)}{\theta}, & u \in (1, 1+\theta] \\
          0, & u > 1 + \theta~.
       \end{cases}
    \]
Note that by construction $H_{\theta}$ is continuous on $[0, \infty)$, and by Assumption (\ref{eq:2brackets}) and the fact that $\mathcal{H}$ has diameter $1$ by definition, $N_2^B(\delta_j, \mathcal{H}, P_2) \le \exp(H_{\theta}(\delta_j))$ for $\theta > 0$. From now on, we fix such a $\theta > 0$, and suppress $\theta$ from our notation. For any $f \in \mathcal{F}$, we have by definition that $f = g\cdot h$ for some $h \in \mathcal{H}$, and so given the bracket $(h_j^L, h_j^U)$, define $f_j^L := g\cdot h_j^L{\bf 1}\{g \ge 0\} + g\cdot h_j^U{\bf 1}\{g < 0\}$, and $f_j^U := g \cdot h_j^U{\bf 1}\{g \ge 0\} + g\cdot h_j^L{\bf 1}\{g < 0\}$, and note that by construction $(f_j^L, f_j^U)$ is a bracket for $f$. Let $f_j = f_j^L$, and let $\mathcal{F}_j = \{f_j : f \in \mathcal{F}\}$, then $|\mathcal{F}_j| \le \exp(H(\delta_j))$ and for every $f \in \mathcal{F}$, $||f - f_j||_2 < \delta_j$. 

By a standard chaining argument:
\[P\left(\sup_{f \in \mathcal{F}} |\nu_n(f)| > \xi\right) \le R_1 + R_2 + R_3~,\]
where
\[R_1 = |\mathcal{F}_0|\sup_{f \in \mathcal{F}}P\left(|\nu_n(f)| > \frac{7}{8}\xi\right)~,\]
\[R_2 = \sum_{j=0}^{N-1}|\mathcal{F}_j||\mathcal{F}_{j+1}|\sup_{f \in \mathcal{F}}P\left(|\nu_n(f_j - f_{j+1})| > \eta_j\right)~,\]
\[R_3 = P\left(\sup_{f \in \mathcal{F}}|\nu_n(f_N - f)| > \frac{\xi}{16} + \eta_N\right)~,\]
where $\{\eta_j\}_j$ are chosen such that $\sum_{j=0}^N\eta_j \le \xi/16$ and will be specified precisely later in the proof. We now choose $\{\delta_j\}_j$, $\{\eta_j\}_j$ and $N$ to make these three terms sufficiently small.

First consider $R_1$. Take $\delta_0$ such that $H(\delta_0) = \xi^2/4$. Then by Hoeffding's inequality,
%NOTE: that the above is an equality is actually important for the very end of the proof!
\[R_1 \le 2|\mathcal{F}_0|\exp\left(-2\left(\frac{7}{8}\xi\right)^2\right) \le 2\exp\left(-\xi^2\right)~.\]
Next, we develop a bound on $R_2$. Since by construction $||f_j - f_{j+1}||_2 \le 2\delta_j$, it follows by repeated applications of Bennet's inequality (see Lemma \ref{lemma:bennet}) that
\[R_2 \le \sum_{j=0}^{N-1}2\exp(2H(\delta_{j+1}))\exp(-\psi_1(\eta_j,n,4\delta_j^2))~,\]
where $\psi_1$ has the properties described in Lemma \ref{lemma:bennet}.
Next, consider $R_3$. Given the construction of $\mathcal{F}_N$ and writing $f = g\cdot h$:
\[|\nu_n(f_N - f)| \le |\nu_n(f_N^U - f_N^L)| + 2\sqrt{n}||f_N^U - f_N^L||_1\]
\[\le |\nu_n(f_N^U - f_N^L)| + 2\sqrt{n}\delta_N^2~,\]
since $||f_N^U - f_N^L||_1 \le ||h_N^U - h_N^L||_1 \le \delta_N^2$ (where here we use the fact that $h_N^U, h_N^L$ take values in $\{0, 1\}$).
Take $\delta_N \le s:= (\xi/(32\sqrt{n}))^{1/2}$, then by the above derivation and Bennet's inequality,
\[R_3 \le P\left(\sup_{f\in\mathcal{F}}|\nu_n(f_N^U - f_N^L)| > \eta_N\right)\]
\[\le 2|\mathcal{F}_N|\exp(-\psi_1(\eta_N,n,\delta^2_N))~.\]
To complete our bounds on $R_2$ and $R_3$ we consider two separate cases. First suppose $\delta_0 \le s$ as defined above. Then by taking $N = 0$ and $\eta_0 = \xi/16$, we have that $R_2 = 0$ and 
\[R_3 \le 2|\mathcal{F}_0|\exp(-\psi_1(\eta_0,n,\delta^2_0))~.\]
Since $\delta_0 \le s$, we have that
\[2\eta_0 = \frac{\xi}{8} = 4\sqrt{n}\left(\frac{\xi}{32\sqrt{n}}\right) \ge 4\sqrt{n}\delta_0^2~,\]
and hence by the properties of $\psi_1$:
\[\psi_1(\eta_0, n, \delta^2_0) \ge \frac{1}{4}\psi_1(2\eta_0,n,\delta^2_0) \ge \frac{1}{4}\eta_0\sqrt{n}~.\]
Using Assumption (\ref{eq:xi_upper}), we can then conclude that
\[\psi_1(\eta_0,n,\delta^2_0) \ge \frac{1}{4}\eta_0\sqrt{n} = \frac{\xi}{64}\sqrt{n} \ge 2\xi^2~.\]
By the definition of $\delta_0$,
\[|\mathcal{F}_0| \le \exp\left(\frac{\xi^2}{4}\right)~,\]
so that putting everything together yields
\[R_2 + R_3 \le 4 \exp(-\xi^2)~.\]
Next consider the case where $\delta_0 > s$. Let $N$ and $\{\delta_j\}_{j=1}^N$ be as in Lemma \ref{lemma:alexander_3.1}, where $t = \delta_0$, and $s$ is as defined above. Let $\eta_j = 8\sqrt{2}\delta_jH(\delta_{j+1})^{1/2}$ for $0 \le j < N$, $\eta_N = 8\sqrt{2}\delta_NH(\delta_{N})^{1/2}$. Then by Lemma \ref{lemma:alexander_3.1}:
\[\sum_{j=0}^N\eta_j = 8\sqrt{2}\sum_{j=0}^N\delta_jH(\delta_{j+1})^{1/2} \le 64\sqrt{2}\int_{s/4}^{\delta_0}H(u)^{1/2}du~.\]
Now, by the definition of $H(\cdot)$, we have that for $0 < s < t$,
    \[ \int_{s}^tH(u)^{1/2}du \le \begin{cases} 
          K^{1/2}\log(1/s), & \ell = 2 \hspace{1mm} \text{and} \hspace{1mm} t \le 1 \\
          2K^{1/2}(\ell-2)^{-1}s^{(2-\ell)/2}, & \ell > 2 
       \end{cases}
    \]
and so by combining this with Assumptions (\ref{eq:xi_upper}) and (\ref{eq:xi_lower}) (with $C_1$ sufficiently large), it can be shown that
\[\sum_{j=0}^N\eta_j \le \frac{\xi}{16}~,\]
and hence our choice of $\{\eta_j\}_j$ is consistent with our construction (note that when $\ell = 2$, the above inequality only applies when $t \le 1$, however we can argue using $\delta_0 \le 1 + \theta$ that $\sum\eta_j \le \xi/16 + C'\theta$ for some constant $C' > 0$, for all $\theta > 0$, and hence our result holds for $P(\sup_{f} |\nu_n(f)| \ge \xi + C'\theta)$ where $\theta > 0$ can be made arbitrarily small). By Assumption (\ref{eq:xi_lower}) (with $C_1$ sufficiently large), it can also be shown that
\[H(s) \le \frac{\xi\sqrt{n}}{16}~,\]
and hence it follows that
\[\left(\frac{\eta_j}{4\delta_j^2\sqrt{n}}\right)^2 < \frac{8H(s)}{ns^2} \le 16~,\]
so that by the properties of $\psi_1$, 
\[\psi_1(\eta_j, n, 4\delta_j^2) \ge \frac{\eta_j^2}{16\delta_j^2}~.\]
Using our bound on $R_2$ we can then conclude that
\[R_2 \le \sum_{j=0}^{N-1}2\exp\left(2H(\delta_{j+1}) - \frac{\eta_j^2}{16\delta_j^2}\right) \le \sum_{j=0}^{N-1}2\exp\left(-4^{j+1}H(\delta_0)\right)~,\]
Similarly, we can argue that
\[R_3 \le 2\exp\left(-4^{N+1}H(\delta_0)\right)~.\]
Putting these together, and using Assumption (\ref{eq:xi_lower}):
\[R_2 + R_3 \le \sum_{j=0}^{\infty}2\exp\left(-4^{j+1}H(\delta_0)\right) \le C\exp(-\xi^2)~,\]
where $C$ is a constant that depends only on $K$ and $\ell$.
\end{proof}

\begin{corollary}\label{cor:tail_ineq}
Let $\{Z_i\}_{i=1}^n$ be a sequence of i.i.d random vectors with distribution $P$. Let $Z = (Z_1,Z_2)$, and let $\mathcal{F}$ be a class of real-valued functions of the form $f(z) = f(z_1,z_2) = g(z)\cdot h(z_2)$, where $h \in \mathcal{H}$, $\mathcal{H}$ is a class of functions with values in $\{0,1\}$, and $g$ is some fixed real-valued function (which may depend on $P$) such that $|g| \le 1$. Suppose $\mathcal{H}$ satisfies Assumption \ref{ass:bracket_entropy}, and suppose that $P_2$, the marginal distribution of $Z_2$, has a density with respect to Lebesgue measure bounded above by $A$. Then there exist positive constants $D_1, D_2, D_3$ (which depend only on $K, r$, and $A$) such that for $n \ge 3$:
 \[P^n\left(\sup_{f \in \mathcal{F}} \left|\frac{1}{\sqrt{n}}\sum_{i=1}^n[f(Z_i) - Ef(Z_i)]\right| > xq_n \right) \le D_3\exp(-x^2q_n^2)~,\]
 for $D_1 \le x \le D_2\sqrt{n}/q_n$, where
\[
q_n = \begin{cases} 
          \log n, & r = 1 \\
          n^{(r - 1)/2(r+1)},& r > 1 
       \end{cases}
\]
\end{corollary}
\begin{proof}
First note that since $P_2$ has density with respect to Lebesgue measure bounded above by $A$, we get that by Assumption \ref{ass:bracket_entropy},
\[H_1^B(\epsilon, \mathcal{H}, P_2) \le C\epsilon^{-r}~,\]
where $C$ is some constant which depends only on $K$ and $A$. Next, since $\mathcal{H}$ consists of $\{0,1\}$-valued functions, any $\epsilon$-bracket for $\mathcal{H}$ in $L_1$ is an $\epsilon^{1/2}$-bracket in $L_2$ and vice versa. Hence we get that
\[H_2^B(\epsilon,\mathcal{H},P_2) \le K'\epsilon^{-2r}~,\]
for some constant $K'$ which depends only on $K$, $r$, and $A$. 
The result then follows immediately by Proposition \ref{prop:alexander_ineq}. 
\end{proof}

\begin{lemma}\label{lem:small_tail_ineq}
Maintain the assumptions of Proposition \ref{prop:upper_bound} with $r < 1$. Let $S_n$ be as in the proof of Proposition \ref{prop:upper_bound}. Then
\[\sup_{P \in \mathcal{P}} E[S_n] = O(1)~.\]
\end{lemma}
\begin{proof}
By definition, $S_n \le \sqrt{n}S_n^{(1)} + S_n^{(2)}$, where
\[S_n^{(1)} = \sup_{G \in \mathcal{G}: ||f_G|| \le n^{-1/(2+2r)}}\left|\frac{1}{n}\sum_{i=1}^n(\tilde{f}_G(Z_i) - E[\tilde{f}_G(Z_i)])\right|~,\]
\[S_n^{(2)} = \sup_{G \in \mathcal{G}: ||f_G|| \ge n^{-1/(2+2r)}}\sqrt{n}\frac{\left|\frac{1}{n}\sum_{i=1}^n(\tilde{f}_G(Z_i) - E[\tilde{f}_G(Z_i)])\right|}{||\tilde{f}_G||^{1-r}}~,\]
with $\tilde{f}_G = \bar{g}\cdot({\bf 1}\{X \in G^*\} - {\bf 1}\{X \in G\})$ and $||\cdot||$ the $L_2(P)$ norm, and we have used the fact that $||\tilde{f}_G||\le 1$. We will use Lemma 5.13 in \cite{geer2000} to bound each of these quantities.
To apply the lemma, let $g$ in her notation be $\tilde{f}_G$ in ours, and $g_0$ in her notation be zero. Set $\alpha = 2r$, $\beta = 0$ in the statement of her lemma. It remains to verify condition (5.40) in her lemma for the class $\mathcal{\widetilde{F}} = \{\tilde{f}_G: G \in \mathcal{G}\}$, but this follows by Assumption \ref{ass:bracket_entropy} by combining the arguments from the proof of Corollary \ref{cor:tail_ineq} and the proof of Proposition \ref{prop:alexander_ineq}. By inequality (5.42) in her lemma:
\[\sup_{P \in \mathcal{P}}n^{1/(1+r)}E[S_n^{(1)}] = O(1)~,\]
and hence since $r < 1$,
\[\sup_{P \in \mathcal{P}}\sqrt{n}E[S_n^{(1)}] = O(1)~.\]
By inequality (5.43),
\[\sup_{P \in \mathcal{P}}E[S_n^{(2)}] = O(1)~.\]
Combining both of these together gives our desired result.
\end{proof}

\begin{lemma}\label{lemma:KL_calculate} Let $P_f$ be specified as in the proof of Proposition \ref{prop:lower_bound}. Then for $f, g \in \mathcal{H}$ such that $f \ne g$:
$$\mathcal{K}(P_f, P_g) \le ||f - g||_1~,$$
where $\mathcal{K}(\cdot, \cdot)$ is the Kullback-Leibler divergence. 
\end{lemma}
\begin{proof}
Let $Q_{f,x}(\cdot)$ denote the probability mass function of $(Y(1),D)|X=x$ under $P_f$ (recall that $Y(0)|X=x$ is degenerate, so we omit it from the calculation). If $f \ne g$, a direct calculation shows that:
$$\mathcal{K}(Q_{f,x},Q_{g,x}) = \frac{1}{2}\log(3)~.$$
%$$\mathcal{K}(Q_{f,x},Q_{g,x}) = Q_{f,x}(-1)\log\left(\frac{Q_{f,x}(-1)}{Q_{g,x}(-1)}\right) + Q_{f,x}(1)\log\left(\frac{Q_{f,x}(1)}{Q_{g,x}(1)}\right) = c_0\log\left(\frac{1+c_0}{1-c_0}\right)~,$$
Hence
$$\mathcal{K}(P_f, P_g) = \int_{\mathcal{X}}\mathcal{K}(Q_{f,x},Q_{g,x}){\bf 1}\{f(x) \ne g(x)\}d\mu = \frac{1}{2}\log{3}||f - g||_1 \le ||f-g||_1~.$$
\end{proof}

\begin{lemma}\label{lemma:bennet}(Bennet's Inequality: see Theorem 2.9 in \cite{boucheron2013})
Let $\{Z_i\}_{i=1}^n$ be a sequence of independent random vectors with distribution $P$. Let $f$ be some function taking values in $[0,1]$ and define 
\[\nu_n(f) := \frac{1}{\sqrt{n}}\sum_{i=1}^n[f(Z_i) - Ef(Z_i)]~.\]
Then for any $\xi \ge 0$,
\[P^n(|\nu_n(f)| > \xi) \le 2\exp(-\psi_1(\xi, n, \alpha))~,\]
where $\alpha = \text{var}(\nu_n(f))$ and 
\[\psi_1(\xi, n, \alpha) = \xi\sqrt{n}h\left(\frac{\xi}{\sqrt{n}\alpha}\right)~,\]
with
\[h(x) = (1 + x^{-1})\log(1 + x) - 1~.\]
Importantly, $\psi$ has the following two relevant properties:
\[\psi_1(\xi,n,\alpha) \ge \psi_1(C\xi,n,\rho\alpha)\ge C^2\rho^{-1}\psi_1(\xi,n,\alpha)~,\]
for $C \le 1$, $\rho \ge 1$, and
\[
  \psi_1(\xi,n,\alpha) \ge
  \begin{cases}
                                   \frac{\xi^2}{4\alpha} & \text{if $\xi < 4\sqrt{n}\alpha$} \\
                                   \frac{\xi}{2}\sqrt{n} & \text{if $\xi \ge 4\sqrt{n}\alpha $} \\
  \end{cases}
\]
\end{lemma}

\begin{lemma}\label{lemma:alexander_3.1}(Lemma 3.1 in \cite{alexander1984})
Let $H:(0,t] \rightarrow \mathbb{R}^{+}$ be a decreasing function, and let $0 < s < t$. Let $\delta_0 := t$, $\delta_{j+1} : = s \vee \sup\{x \le \delta_j/2: H(x) \ge 4H(\delta_j)\}$ for $j \ge 0$, and $N := \min\{j:\delta_j = s\}$. Then
\[\sum_{j=0}^N\delta_jH(\delta_j)^{1/2} \le 8\int_{s/4}^tH(u)^{1/2}du~.\]
\end{lemma}

\end{small}
\clearpage
\nocitesupplement{*}
\bibliographystylesupplement{te}
\bibliographysupplement{supp_references.bib}

\end{document}